\documentclass[final]{siamltex}

\usepackage{amsfonts}
\usepackage{amsmath}
\usepackage{graphicx}

\newcommand{\RR}{{\mathbb{R}}}
\newcommand{\CC}{{\mathbb{C}}}

\newtheorem{remark}[theorem]{Remark}
\newtheorem{example}[theorem]{Example}

\begin{document}

\thispagestyle{empty}
\bibliographystyle{siam}

%\title{Decay bounds for functions of banded matrices
%and of sparse matrices with Kronecker structure}
\title{Decay bounds for functions of matrices with banded or
Kronecker structure}
\author{Michele Benzi\thanks{Department of Mathematics and Computer
Science, Emory University, Atlanta, Georgia 30322, USA
(benzi@mathcs.emory.edu). The work of this author was supported
by National Science Foundation grants
DMS-1115692 and DMS-1418889.}  \and 
Valeria Simoncini\thanks{Dipartimento di Matematica, Universit\`a di Bologna,
Piazza di Porta S.~Donato 5, I-40127 Bologna, Italy (valeria.simoncini@unibo.it).
The work of this author was partially supported by the FARB12SIMO grant
of the Universit\`a di Bologna.}
}

\maketitle

\markboth{{\sc M.~Benzi and V.~Simoncini}}{Decay bounds for matrix
functions}

\begin{abstract}
We present decay bounds for a broad class of Hermitian matrix functions
where the matrix argument is banded or a Kronecker sum of banded matrices.
Besides being significantly tighter than previous estimates, 
the new bounds 
closely capture the actual (non-monotonic)
decay behavior of the entries of functions of matrices with Kronecker sum
structure. We also discuss 
extensions to more general sparse matrices.
\end{abstract}



%%%%%%%%%%%%%%%%%%%%%%%%%%%%%%%%%%%%%%%%%%%%%
\section{Introduction}
The decay behavior of the entries of functions of banded and sparse matrices
has attracted considerable interest over the years.
It has been known for some time that if $A$ is a banded Hermitian matrix and $f$ is
a smooth function with no singularities in a neighborhood of the spectrum of $A$,
then the entries in $f(A)$ usually exhibit rapid decay in magnitude away from the
main diagonal. The decay rates are typically exponential,
with even faster decay in the case of entire functions.  

The interest for the decay behavior of matrix functions stems largely from its
importance for a number of applications including numerical analysis
\cite{Benzi.Golub.99,CanutoSimonciniVeraniJOMP.14,DelLopPel05,Demko,EP88,Meurant,ye13},
harmonic analysis \cite{Bask1,Gro10,Jaffard},
quantum chemistry \cite{BBR13,BM12,Lin14,Shao}, signal processing
\cite{KSW,Strohmer}, quantum information theory \cite{CE06,CEPD06,ECP10},
 multivariate statistics \cite{Aune}, queueing models \cite{Bini05}, control
of large-scale dynamical systems \cite{Haber14}, quantum dynamics \cite{Giscard14},
random matrix theory \cite{Molinari}, and others.
The first case to be analyzed in detail was that of $f(A) = A^{-1}$,
see \cite{Demko,DMS,EP88,Kershaw}. In these papers one can find exponential
decay bounds for the entries of the inverse of banded matrices.
A related, but quite distinct line of research concerned
the study of inverse-closed matrix algebras, where the decay behavior
%(be it exponential or algebraic)
in the entries of a (usually infinite) matrix
$A$ is ``inherited" by the entries of
$A^{-1}$. Here we mention \cite{Jaffard}, 
where it was observed that a similar decay behavior occurs for the entries of
$f(A) = A^{-1/2}$, as well as
\cite{Bask1,Bask2,Gro10,GroLei06,KSW}, among others.

The study of the decay behavior for 
general analytic functions of banded matrices, including the
important case of the matrix exponential, was initiated in \cite{Benzi.Golub.99,iserles}
and continued for possibly non-normal matrices and general sparsity patterns 
in \cite{Benzi2007}; further contributions in these directions include
\cite{BB14,DelLopPel05,mastronardi,Shao}. 
Collectively, these papers have largely elucidated the question of when one can expect
exponential decay in the entries of $f(A)$, in terms of conditions that 
the function $f$ and the matrix $A$ must satisfy. Some of these papers also
address the important problem of when the rate of decay is asymptotically independent
of the dimension $n$ of the problem, a condition that allows, at least in principle,
for the approximation of $f(A)$ with a computational cost scaling linearly in $n$
(see, e.g., \cite{BBR13,Benzi2007,BM12}).

A limitation of these papers is that they provide decay bounds for the entries
of $f(A)$ that are often pessimistic and may not capture the correct, non-monotonic
decay behavior actually observed in many situations of practical interest. A first step
to address this issue was taken in \cite{CanutoSimonciniVeraniLAA.14}, where new
bounds for the inverses of matrices that are Kronecker sums of banded matrices 
(a kind of structure of considerable importance in the numerical solution of PDE
problems) were obtained; see also \cite{Meurant} for an early such
analysis for a special class of matrices, and \cite{mastronardi} for functions
of multiband matrices. 

In this paper we build on the work in \cite{CanutoSimonciniVeraniLAA.14}
to investigate the decay behavior in (Hermitian) matrix functions where the
matrix is a Kronecker sum of banded  matrices. We also present new bounds
for functions of banded (more generally, sparse)
 Hermitian matrices. For certain broad
classes of analytic functions that frequently arise in applications 
(including as special cases the resolvent, the inverse square root, and the exponential)
we obtain improved decay bounds that capture much more closely the actual decay
behavior of the matrix entries than previously published bounds. 
A significant difference with previous work in this area is that our bounds
are expressed in integral form, and in order to apply the bounds to specific
matrix functions it may be necessary to evaluate these integrals numerically.

The paper is organized as follows. In section~\ref{sec:pre} we provide basic
definitions and material from linear algebra and analysis utilized in the rest
of the paper. In section~\ref{sec:prev} we briefly recall earlier work on
decay bounds for matrix functions. New decay results for functions of
banded matrices are given in section~\ref{sec:banded}.
Generalizations to more general sparse matrices
are briefly discussed in section~\ref{sec:ext}. Functions
of matrices with Kronecker sum structure are treated in section~\ref{sec:Kron}.
Conclusive remarks are given in section~\ref{sec:Conc}.

\section{Preliminaries}\label{sec:pre}

In this section we give some basic definitions and background
information on the types of matrices and functions considered in
the paper.

\subsection{Banded matrices and Kronecker sums}

We begin by recalling
two standard definitions. 

\begin{definition}
We say that a matrix $M\in \CC^{n\times n}$ is $\beta$-banded if its 
entries $M_{ij}$ satisfy 
$M_{ij} = 0$ for $|i-j|>\beta$.
\end{definition}

\begin{definition}
Let $M_1, M_2\in \CC^{n\times n}$.
We say that a matrix ${\cal A}\in \CC^{n^2\times n^2}$ is the {\em Kronecker sum}
of $M_1$ and $M_2$ if
\begin{eqnarray}\label{eqn:kron}
{\cal A} = M_1\oplus M_2 := M_1\otimes I + I\otimes M_2\,,
\end{eqnarray}
where $I$ denotes the $n\times n$ identity matrix.
\end{definition}

In this paper we will be especially concerned with the case $M_1=M_2=M$,
where $M$ is $\beta$-banded and Hermitian positive definite (HPD).
In this case $\cal A$ is also HPD.

The definition of Kronecker sum can easily be extended to three or more
matrices. For instance, we can define
$$
{\cal A} = M_1\oplus M_2 \oplus M_3 :=
M_1\otimes I\otimes I +
I\otimes M_2\otimes I +
I\otimes I\otimes M_3  .
$$

The Kronecker sum of two matrices is well-behaved under matrix exponentiation.
Indeed,
the following relation holds (see, e.g., \cite[Theorem 10.9]{Higham2008}):
\begin{eqnarray}\label{eqn:exp_kron}
%\exp(M_1 \otimes I_m + I_m \otimes M_2) =
\exp(M_1\oplus M_2) = 
\exp(M_1)\otimes \exp(M_2) .
\end{eqnarray}

Similarly, the following matrix trigonometric identities hold
for the matrix sine and cosine \cite[Theorem 12.2]{Higham2008}:
\begin{eqnarray}\label{eqn:sin_kron}
\sin(M_1\oplus M_2) = \sin(M_1)\otimes \cos(M_2) + \cos(M_1)\otimes \sin(M_2)
\end{eqnarray}
and
\begin{eqnarray}\label{eqn:cos_kron}
\cos(M_1\oplus M_2) = \cos(M_1)\otimes \cos(M_2) - \sin(M_1)\otimes \sin(M_2).
\end{eqnarray}

As we will see, identity (\ref{eqn:exp_kron}) 
will be useful in extending decay
results for functions of banded matrices to functions of matrices with
Kronecker sum structure.

%%%%%%%%%%%%%%%%%%%%%%%%%
\subsection{Classes of functions defined by integral transforms}\label{sec:classes}
We will be concerned with analytic functions of matrices. 
It is well known that if $f$ is a function analytic in a domain $\Omega \subseteq \CC$
containing the spectrum of a matrix $A\in \CC^{n\times n}$, then
\begin{eqnarray}\label{eqn:contour}
f(A) = \frac {1}{2\pi i}\int_{\Gamma} f(z)(zI - A)^{-1} {\rm d}z\,,
\end{eqnarray}
where $i= \sqrt{-1}$ is the imaginary unit and $\Gamma$ is any simple closed curve
surrounding the eigenvalues of $A$ and entirely contained in $\Omega$, 
oriented counterclockwise.

Our main
results concern certain analytic functions that can be represented as integral
transforms of measures, in particular, {\em strictly completely monotonic
functions} (associated with the Laplace--Stieltjes transform) and {\em Markov functions}
(associated with the Cauchy--Stieltjes transform). Here we briefly review some
basic properties of these functions and the relationship between the
two classes.
We begin with the following definition (see \cite{Widder.46}).

\vskip 0.01in

\begin{definition}
Let $f$ be defined in the interval $(a,b)$ where $-\infty \le a < b \le +\infty$.
Then, $f$ is said to be {\em completely monotonic} in $(a,b)$ if 
$$(-1)^{k}f^{(k)} (x) \ge 0 \quad {\rm for\ all} \quad a < x < b \quad {\rm and\ all}
\quad k=0,1,2,\ldots $$
Moreover, $f$ is said to be {\em strictly completely monotonic}
in $(a,b)$ if 
$$(-1)^{k}f^{(k)} (x) > 0 \quad {\rm for\ all} \quad a < x < b \quad {\rm and\ all}
\quad k=0,1,2,\ldots $$
\end{definition}

Here $f^{(k)}$ denotes the $k$th derivative of $f$, with $f^{(0)}\equiv f$.
It is shown in \cite{Widder.46} that if $f$ is completely monotonic
in $(a,b)$, it can be extended to an analytic function in the open disk
$|z - b| < b - a$ when $b$ is finite. When $b=+\infty$, $f$ is analytic in
$\Re(z) > a$. Therefore, for each $y\in (a,b)$ we have that $f$ is analytic
in the open disk $|z - y| < R(y)$, where $R(y)$ denotes the radius of convergence
of the power series expansion
of $f$ about the point $z=y$. Clearly, $R(y) \ge y-a$ for $y\in (a,b)$.
%Making the change of variable $z=y-\zeta$ and writing
%$$f(y - \zeta) = \sum_{k=0}^{\infty} b_k(y)\zeta^k, \quad |\zeta| < R(y),$$
%we see that the coefficients are given by
%$$b_k(y) = \frac{(-1)^kf^{(k)}(y)}{k!}, \quad k=0,1,2,\ldots $$
%Hence, if $f$ is completely monotonic in $(a,b)$ and $y\in(a,b)$, the coefficients
%$b_k(y)$ are all nonnegative.

In \cite{Bernstein.29} Bernstein proved that a function $f$ is completely monotonic
in $(0,\infty)$ if and only if $f$ is the Laplace--Stieltjes transform
of $\alpha (\tau)$;
\begin{equation}\label{bern}
f(x) = \int_0^\infty {\rm e}^{-\tau x} {\rm d}\alpha(\tau),
\end{equation}
where $\alpha (\tau)$ is nondecreasing and
the integral in (\ref{bern}) converges for all $x>0$. Moreover,
under the same assumptions $f$ can be extended
to an analytic function on the positive half-plane $\Re(z) > 0$.
A refinement of this result (see \cite{Dub40}) states that
$f$ is strictly completely monotonic
in $(0,\infty)$ if it is completely monotonic there and moreover the
function $\alpha (\tau)$ has at least one positive point of increase, that is,
there exists a $\tau_0 > 0$ such that $\alpha(\tau_0+\delta) > \alpha(\tau_0)$ 
for any $\delta >0$.

Prominent examples of strictly completely monotonic functions include (see \cite{Varga.68}):

\begin{enumerate}
\item $f_1(x) = 1/x = \int_0^\infty {\rm e}^{-x\tau} d\alpha_1(\tau)$ for $x>0$,
where $\alpha_1(\tau) = \tau$ for $\tau\ge 0$.
\item $f_2(x) = {\rm e}^{-x} = \int_0^\infty {\rm e}^{-x\tau} d\alpha_2(\tau)$ for $x>0$,
where $\alpha_2(\tau) = 0$ for $0\le \tau < 1$ and $\alpha_2(\tau) = 1$ for $\tau\ge 1$.
\item $f_3(x) = (1 - {\rm e}^{-x})/x = \int_0^\infty {\rm e}^{-x\tau} d\alpha_3(\tau)$ 
for $x>0$,
where $\alpha_3 (\tau) = \tau$ for $0\le \tau \le 1$, and $\alpha_3(\tau) = 1$ for $\tau\ge 1$.
\end{enumerate}

\vspace{0.1in}
Other examples include the functions $x^{-\sigma}$ (for any $\sigma > 0$),
$\log(1+1/x)$ and $\exp(1/x)$, 
all strictly completely monotonic on $(0,\infty)$. 
Also, products and positive linear combinations
of strictly completely monotonic functions are strictly completely monotonic, 
as one can readily check.

A closely related class of functions is given by the Cauchy--Stieltjes (or 
Markov-type) functions, which can be written as 
\begin{eqnarray}\label{eqn:markov}
f(z) = \int_\Gamma \frac {{\rm d}\gamma(\omega)} {z-\omega},
\quad z\in \CC \setminus \Gamma\,,
\end{eqnarray}
where $\gamma$ is a (complex) measure supported on a closed set $\Gamma \subset \CC$ 
and the integral is absolutely convergent.
%\begin{eqnarray}\label{eqn:markov}
In this paper we are especially interested in the special case $\Gamma = (-\infty, 0]$
so that
$$f(x) = \int_{-\infty}^0 \frac {{\rm d}\gamma(\omega)} {x - \omega}, 
\quad x\in \CC \setminus (-\infty, 0]\,,
$$
where $\gamma$ is now a (possibly signed) real measure.
The following functions, which occur in various applications (see, e.g., \cite{Guettel.13}
and references therein),
fall into this class:
\begin{eqnarray*}
&& z^{-\frac 1 2} = \int_{-\infty}^0 \frac 1 {z-\omega} \frac 1 
{\pi \sqrt{-\omega}} {\rm d}\omega,
\\
&& \frac{{\rm e}^{-t\sqrt{z}}-1}{z} = \int_{-\infty}^0 \frac 1 {z-\omega} \frac 
{\sin(t\sqrt{-\omega})}{-\pi \omega} {\rm d}\omega,
\\
&& \frac{\log(1+z)}{z} = \int_{-\infty}^{-1} \frac 1 {z-\omega} \frac 1 
{(-\omega)} {\rm d}\omega.
\end{eqnarray*}

The two classes of functions just introduced overlap.
Indeed, it is easy to see (e.g., \cite{Merkle}) that functions
of the form
$$f(x) = \int_0^{\infty} \frac {{\rm d}\mu(s)}{ x + \omega },$$
with $\mu$ a positive measure, are strictly completely monotonic
on $(0,\infty)$; 
but every such function can also
be written in the form
$$f(x) = \int_{-\infty}^0 \frac {{\rm d}\gamma(\omega)}{ x - \omega}, 
\quad \gamma(\omega) = - \mu(-\omega),$$
and therefore it is a Cauchy--Stieltjes function.  We note that $f(x) = \exp(-x)$
is an example of a function that is strictly completely monotonic but not 
a Cauchy--Stieltjes function. 

In the rest of the paper, the term {\em Laplace--Stieltjes function}
will be used to denote a function that is strictly completely monotonic on $(0,\infty)$.

\section{Previous work} \label{sec:prev}
%Describe the work of Demko et al.
%Mention that the Iserles bounds for the exponential of a $\beta$-banded matrix 
%with $\beta > 1$ are only valid for sufficiently large distances from the
%main diagonal.
In this section we briefly review some previous decay results from the
literature.

Given a $n\times n$ Hermitian positive definite $\beta$-banded matrix $M$, it was shown in
\cite{DMS} that
\begin{eqnarray}\label{eqn:demko}
|(M^{-1})_{ij} | \le C q^{\frac{|i-j|}{\beta}} 
\end{eqnarray}
for all $i,j=1,\ldots ,n$, where $q =(\sqrt{\kappa}-1)/(\sqrt{\kappa}+1)$,
$\kappa$ is the spectral condition number of $M$,
$C = \max\{1/\lambda_{\rm min}(M), \hat C\}$, and
$\hat C= (1+\sqrt{\kappa})^2/(2\lambda_{\max}(M))$. In this
bound the diagonal elements of $M$ are assumed not to be
greater than one, which can always be satisfied by dividing $M$
by its largest diagonal entry, after which the bound (\ref{eqn:demko}) 
will have to be multiplied
by its reciprocal. The bound is known to be sharp, in the sense
that it is attained for a certain tridiagonal Toeplitz matrix.
We mention that (\ref{eqn:demko}) is also valid for 
infinite and bi-infinite matrices as long as they have finite
condition number, i.e., both $M$ and $M^{-1}$ are bounded.
 Using the identity $M^{-1} = (M^*M)^{-1}M^*$, simple decay bounds
were also obtained in \cite{DMS} for non-Hermitian matrices.

Similarly,
if $M$ is $\beta$-banded and Hermitian and $f$ is analytic on a region of the
complex plane containing
the spectrum $\sigma (M)$ of $M$, then there exist positive constants $C$ and
$q<1$ such that
\begin{eqnarray}\label{eqn:BG}
|(f(M))_{ij} | \le C q^{\frac{|i-j|}{\beta}},
\end{eqnarray}
where $C$ and $q$ can be expressed in terms of the parameter of a certain
ellipse surrounding $\sigma (M)$ and of the maximum modulus of $f$ on this ellipse; see
\cite{Benzi.Golub.99}. The bound (\ref{eqn:BG}), in general, is not
sharp; in fact, since there are infinitely many ellipses containing
$\sigma(M)$ in their interior and such that $f$ is analytic inside the
ellipse and continuous on it, one should think of (\ref{eqn:BG}) as a
parametric family of bounds rather than a single bound. By tuning the 
parameter of the ellipse one can obtain different bounds, usually
involving a trade-off between the values of $C$ and $q$.
This result was extended in \cite{Benzi2007} to the case where $M$ is
a sparse matrix with a general sparsity pattern, using the graph distance
instead of the distance from the main diagonal; see also \cite{CE06,Jaffard} and
section \ref{sec:ext} below.
Similar bounds for analytic functions of non-Hermitian matrices can be found
in \cite{BB14,Benzi2007}.

Practically all of the above results consist of exponential decay bounds on the
magnitude of the entries of $f(M)$. However, for entire functions
the actual decay is typically superexponential, rather than exponential.
Such bounds have been obtained by Iserles for the exponential of
a tridiagonal matrix in \cite{iserles}. This paper also presents
superexponential decay bounds for the exponential of banded matrices,
but the bounds only apply at sufficiently large distances from
the main diagonal. None of these bounds require $M$ to be Hermitian.
Superexponential decay bounds for the exponential
of certain infinite tridiagonal skew-Hermitian matrices
arising in quantum mechanical computations have been
recently obtained in \cite{Shao}.

\section{Decay estimates for functions of a banded matrix} \label{sec:banded}
In this section we present new decay bounds for functions of matrices 
$f(M)$ where $M$ is a banded, Hermitian and positive definite. 
First, we make use of an important result from \cite{HocLub97} to obtain
decay bounds for the entries of the exponential of a banded, Hermitian, positive
semidefinite matrix $M$.  
This result will
then be used to obtain bounds or estimates on the entries of $f(M)$, where
$f$ is strictly completely monotonic.
In a similar manner, we will obtain bounds or estimates on the entries of $f(M)$
where $f$ is a Markov function by making use of the classical bounds of
Demko et al.~\cite{DMS} for the entries of the inverses of banded positive
definite matrices.

In section \ref{sec:Kron} we will use these results to obtain bounds
for matrix functions $f({\cal A})$, 
where $\cal A$ is a Kronecker sum of banded matrices and $f$ 
belongs to one of the two above-mentioned classes of functions.

\subsection{The exponential of a banded Hermitian matrix}
We first recall (with a slightly different notation)
an important result due to Hochbruck and Lubich \cite{HocLub97}.
Here the $m$ columns of $V_m\in \CC^{n\times n}$ form an orthonormal basis for the Krylov
subspace $K_m(M,v)={\rm span}\{v, Mv, \ldots, M^{m-1}v\}$ with $\|v\|=1$, and $H_m = V_m^* M V_m$.

\begin{theorem}\label{th:HL}
Let $M$ be a Hermitian positive semidefinite matrix with eigenvalues
in the interval $[0,4\rho]$. Then the error in the Arnoldi approximation
of $\exp(\tau M) v$ with $\|v\|=1$, namely % and $\rho\tau>1$, namely,
$\varepsilon_m:= \|\exp(-\tau M) v - V_m \exp(-\tau H_m) e_1 \|$,
 is bounded in the following ways:
\begin{enumerate}
\item[i)] $\varepsilon_m \le
10 \exp(-m^2/(5\rho\tau))$, for $\rho\tau\ge 1$ and $\sqrt{4\rho\tau}\le m \le 2\rho\tau$
\item[ii)] $\varepsilon_m \le
10 (\rho\tau)^{-1} \exp(-\rho\tau) \left ( \frac{{\rm e}\rho\tau}{m}\right)^m$ for
$m\ge 2\rho\tau$.
\end{enumerate}
\end{theorem}

\vskip 0.01in

With this result we can establish bounds for the entries of the
exponential of a banded Hermitian matrix.

\vskip 0.01in

\begin{theorem}\label{th:boundexp}
Let $M$ be as in Theorem \ref{th:HL}. Assume in addition 
that $M$ is $\beta$-banded. Then, with
the notation of Theorem \ref{th:HL} and for $k\ne t$: %,  and for $\rho\tau>1$: %$k\ne t$:
\begin{enumerate}
\item[i)] For $\rho\tau\ge 1$ and $\sqrt{4\rho\tau}\le |k-t|/\beta\le 2\rho\tau$,
$$
| (\exp(-\tau M) )_{kt}|
\le
10 \exp\left(-\frac{(|k-t|/\beta)^2}{5 \rho\tau}\right) ;
$$
\item[ii)]
For $|k-t|/\beta \ge 2\rho\tau$,
$$
| (\exp(-\tau M) )_{kt}|
\le
10 \frac{\exp(-\rho\tau)}{\rho\tau}
\left ( \frac{{\rm e}\rho\tau}{\frac{|k-t|}{\beta}}\right)^{\frac{|k-t|}{\beta}} .
$$
\end{enumerate}
\end{theorem}

\begin{proof}
We first note that an element of the Krylov subspace $K_m(M,v)$ is a polynomial
in $M$ times a vector, so that $V_m \exp(-\tau H_m) e_1 = p_{m-1}(\tau M) v$ for some polynomial
$p_{m-1}$ of degree at most $m-1$. Because $M$ is Hermitian and
$\beta$-banded, 
the matrix $p_{m-1}(\tau M)$ is
at most $(m-1)\beta$-banded.

Let now $k,t$ with $k\ne t$ be fixed, and write
$|k-t| = (m-1)\beta + s$ for some $m\ge 1$ and $s=1, \ldots, \beta$; in particular,
we see that $(p_{m-1}(\tau M))_{kt}=0$, moreover
$|k-t|/\beta \le m$. Consider first case ii). If $m \ge 2\rho\tau$,
for $v=e_t$ we obtain
\begin{eqnarray*}
| (\exp(-\tau M) )_{kt}| & = &
| (\exp(-\tau M) )_{kt} - (p_{m-1}(\tau M))_{kt}| \\
&= &
| e_k^T ( \exp(-\tau M) e_t - p_{m-1}(\tau M) e_t)|  \\
&\le &
\|\exp(-\tau M) e_t - p_{m-1}(\tau M) e_t \| \\
&\le&
10 (\rho\tau)^{-1} \exp(-\rho\tau)
%\left ( \frac{e\rho\tau}{|k-t|/\beta}\right)^{\frac{|k-t|}{\beta}} ,
\left ( \frac{e\rho\tau\beta}{|k-t|}\right)^{\frac{|k-t|}{\beta}} ,
\end{eqnarray*}
where in the last inequality Theorem~\ref{th:HL} was used for
$m |k-t|/\beta\ge 2\rho\tau$.
An analogous result is obtained for $m$ in the finite interval, so as to verify i).
\end{proof}

%Note that Theorem \ref{th:boundexp} assumes that the spectrum of $M$ is contained in
%$[0,4\rho]$ and in particular, that $M$ is only positive semidefinite. For a positive
%definite $M$, we can apply the theorem to $M-\lambda_{\min} I$ since
%$\exp(M) = \exp(\lambda_{\min})\exp(M-\lambda_{\min} I)$.

As remarked in \cite{HocLub97}, the restriction to positive semidefinite $M$ 
leads to no loss of generality, since a shift from $M$ to $M+\delta I$ entails
a change by a factor ${\rm e}^{\tau \delta}$ in the quantities of interest.

We also notice that in addition to Theorem \ref{th:HL}
 other asymptotic bounds exist for estimating the
error in the exponential function with Krylov subspace approximation; see, 
e.g., \cite{DK1, DK2}. An advantage of Theorem \ref{th:HL} is that it provides
explicit upper bounds, which can then be easily used for our purposes. 

\begin{example}\label{ex:exp}
{\rm
Figure \ref{fig:expM} shows the behavior of the bound in Theorem \ref{th:boundexp}
for two typical matrices. The plot on the left refers to the tridiagonal
matrix $M={\rm tridiag}(-1,4,-1)$ ($\beta=1$) of size $n=200$, with $\tau=4$, so that
$\tau\rho \approx 3.9995$. The $t$th column with $t=127$ is reported, and only the
values above $10^{-60}$ are shown.
The plot on the right refers to the pentadiagonal matrix
$M={\rm pentadiag}(-0.5,-1,4,-1,-0.5)$ ($\beta=2$) of size $n=200$, with
$\tau=4$, so that $\tau\rho \approx 4.4989$. The same column $t=127$ is shown.
Note the superexponential decay behavior.
In both cases, the estimate seems to be rather sharp.
}
\end{example}

\begin{figure}[t]
\centering
\includegraphics[width=2.5in,height=2.5in]{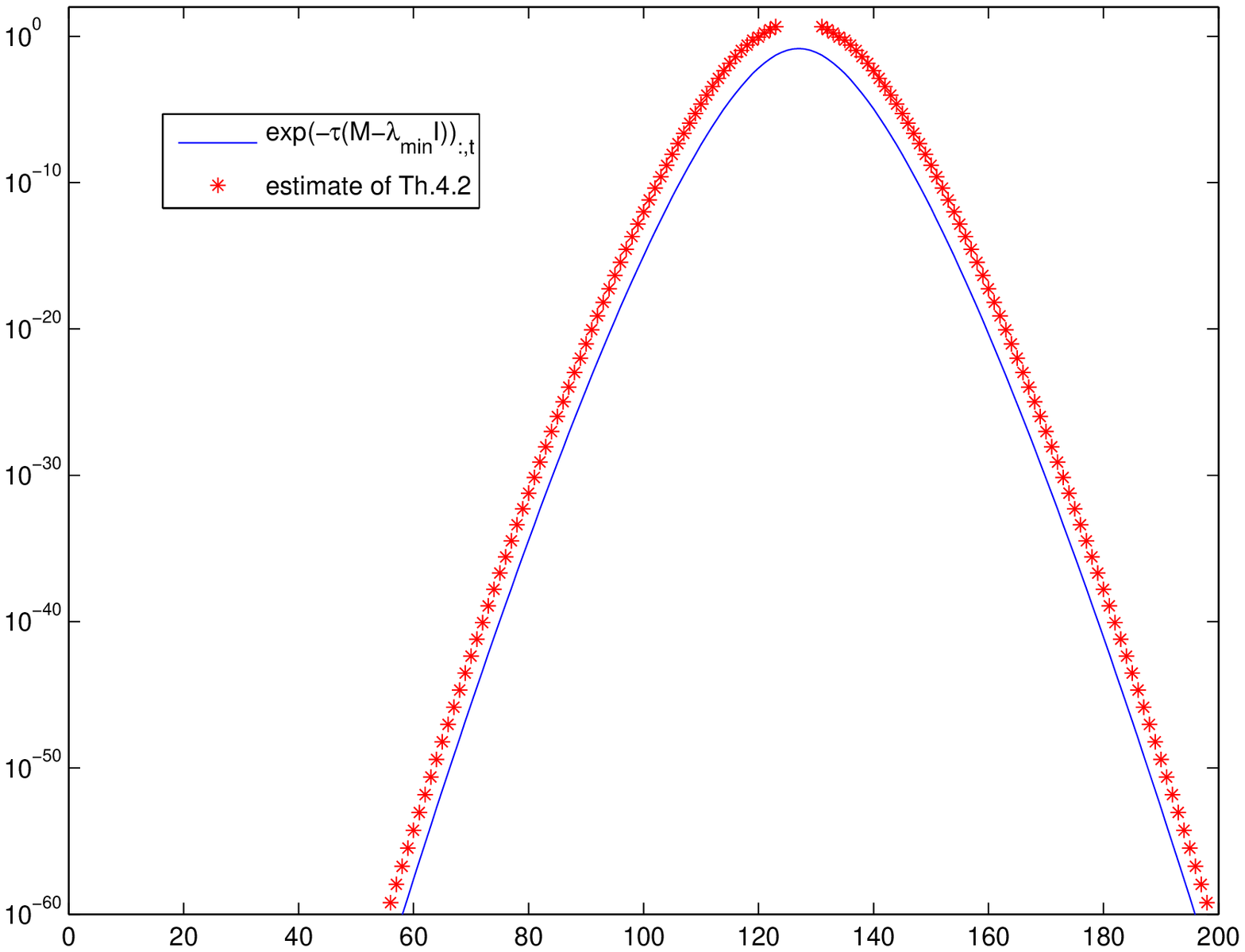}
\includegraphics[width=2.5in,height=2.5in]{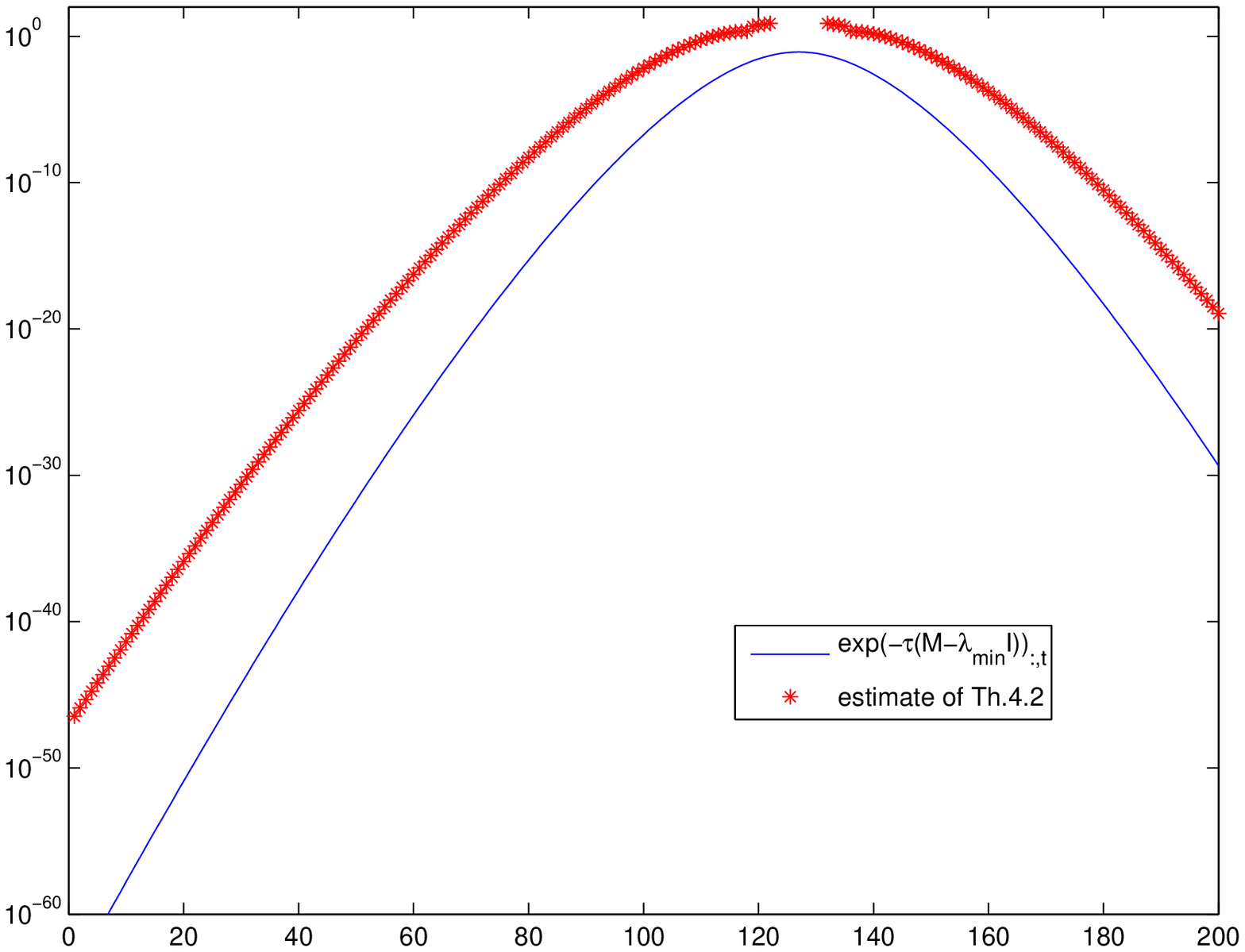}
\caption{Example \ref{ex:exp}. Bounds for $|\exp(-\tau(M-\lambda_{\min}I))|_{:,t}$, $t=127$,
using Theorem \ref{th:boundexp}. $M$ of
size $n=200$ and $\tau=4$.  
Left: $M={\rm tridiag}(-1,4,-1)$.  Right: $M={\rm pentadiag}(-0.5,-1,4,-1,-0.5)$.
Logarithmic scale.
\label{fig:expM}}
\end{figure}

\subsection{Bounds for Laplace--Stieltjes functions}
By exploiting the connection between the exponential function and 
Laplace--Stieltjes functions,
we can apply Theorem \ref{th:boundexp} to obtain bounds or
estimates for the entries of Laplace--Stieltjes matrix functions.

\begin{theorem}\label{th:LS}
Let $M=M^*$ be $\beta$-banded and positive definite, 
and let $\widehat M = M-\lambda_{\min} I$, 
with the spectrum of $\widehat M$
contained in $[0,4\rho]$.
Assume $f$ is a Laplace--Stieltjes function, so that it can be written in the form
$f(x) = \int_0^\infty {\rm e}^{-x\tau} {\rm d}\alpha(\tau)$. 
Then, with the notation and assumptions of Theorem \ref{th:boundexp} and for 
$|k-t|/\beta\ge 2$: %$k\ne t$:
\begin{eqnarray}
|f(M)|_{k,t} &\le& \int_0^\infty \exp(-\lambda_{\min}\tau) 
|(\exp(-\tau\widehat M))_{k,t}| {\rm d}\alpha(\tau)  \nonumber \\
&\le & 10 \int_0^{\frac{|k-t|}{2\rho\beta}} \exp(-\lambda_{\min}\tau)
\frac{\exp(-\rho\tau)}{\rho\tau}
\left ( \frac{{\rm e}\rho\tau}{\frac{|k-t|}{\beta}}\right)^{\frac{|k-t|}{\beta}} 
{\rm d}\alpha(\tau) \label{eqn:LSbound}  \\
&& \quad
 +10 \int_{\frac{|k-t|}{2\rho\beta}}^{\frac{|k-t|^2}{4\rho\beta^2}}
 \exp(-\lambda_{\min}\tau)
 \exp\left(-\frac{(|k-t|/\beta)^2}{5 \rho\tau}\right) {\rm d}\alpha(\tau)  \nonumber \\
&& + \int_{\frac{|k-t|^2}{4\rho\beta^2}}^\infty \exp(-\lambda_{\min}\tau) 
(\exp(-\tau\widehat M))_{k,t} {\rm d}\alpha(\tau) = I + II + III. \nonumber
\end{eqnarray}
\end{theorem}
%The third addend does not significantly contribute to the estimate (TO BE CHECKED)

\vskip 0.01in

In general, these integrals may have to be evaluated numerically.
We observe that
in the above bound, the last term (III) does not significantly contribute
provided that $|k - t|$ is sufficiently large while $\rho$ and $\beta$
are not too large. 
% Hence, estimates (as opposed to guaranteed upper bounds)
%may be obtained by dropping the last term for $|k - t|\gg 1$.

As an illustration, consider the function
$f(x) = 1/\sqrt{x}$. For this function we have 
$\alpha(\tau) = \frac 1{\sqrt{\tau}} 
\Gamma(-\frac 1 2+1) = \sqrt{\pi/\tau}$ with $\tau \in (0,\infty)$.
We have
\begin{eqnarray*}
I + II &= &
10 \sqrt{\pi}  \int_0^{\frac{|k-t|}{2\rho\beta}}
 \frac{\exp(-\lambda_{\min}\tau)}{\tau\sqrt{\tau}}
\frac{\exp(-\rho\tau)}{\rho\tau}
\left ( \frac{{\rm e}\rho\tau}{\frac{|k-t|}{\beta}}\right)^{\frac{|k-t|}{\beta}} {\rm d}\tau \\
&& +
 10\sqrt{\pi} \int_{\frac{|k-t|}{2\rho\beta}}^{\frac{|k-t|^2}{4\rho\beta^2}}
 \frac{\exp(-\lambda_{\min}\tau)}{\tau\sqrt{\tau}}
 \exp\left(-\frac{(|k-t|/\beta)^2}{5 \rho\tau}\right)  {\rm d}\tau ,
\end{eqnarray*}
while
$$
III \le 
 \sqrt{\pi} \int_{\frac{|k-t|^2}{4\rho\beta^2}}^\infty 
\frac{\exp(-\lambda_{\min}\tau) }{\tau\sqrt{\tau}}
{\rm d}\tau .
$$

Figure \ref{fig:LS-1/2} shows two typical bounds for the entries of
$M^{-\frac 1 2}$ for the same matrices $M$ considered in Example \ref{ex:exp}.
The integrals $I$ and $II$ and the one appearing in the upper bound for $III$
have been evaluated accurately
using the built-in Matlab function {\tt quad}.
% with a {\color{red} default} error tolerance of {\color{red} $10^{-6}$}. 
Note that the decay is now exponential.
In both cases, the decay is satisfactorily captured. 

\begin{figure}[thb]
\centering
\includegraphics[width=2.5in,height=2.5in]{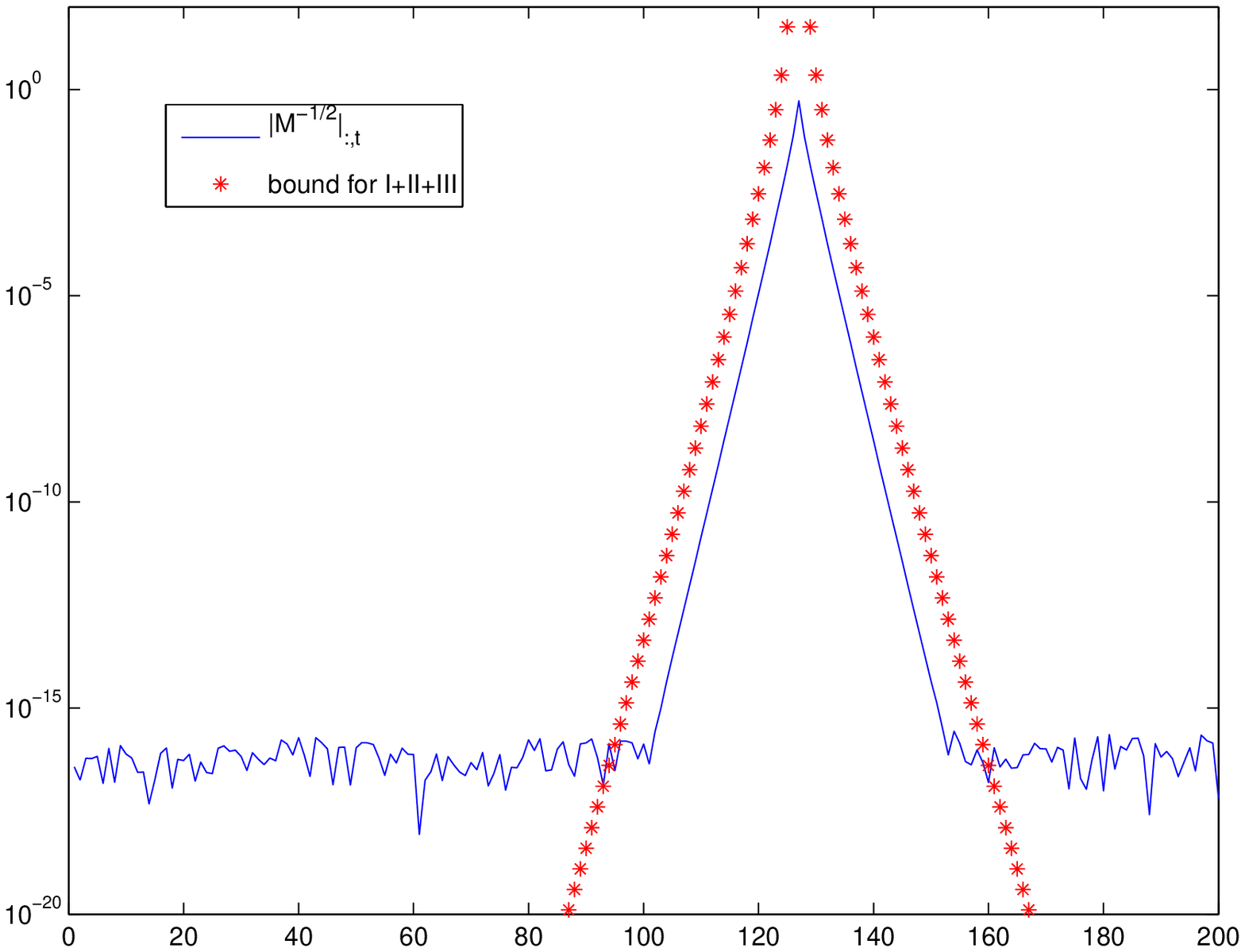}
\includegraphics[width=2.5in,height=2.5in]{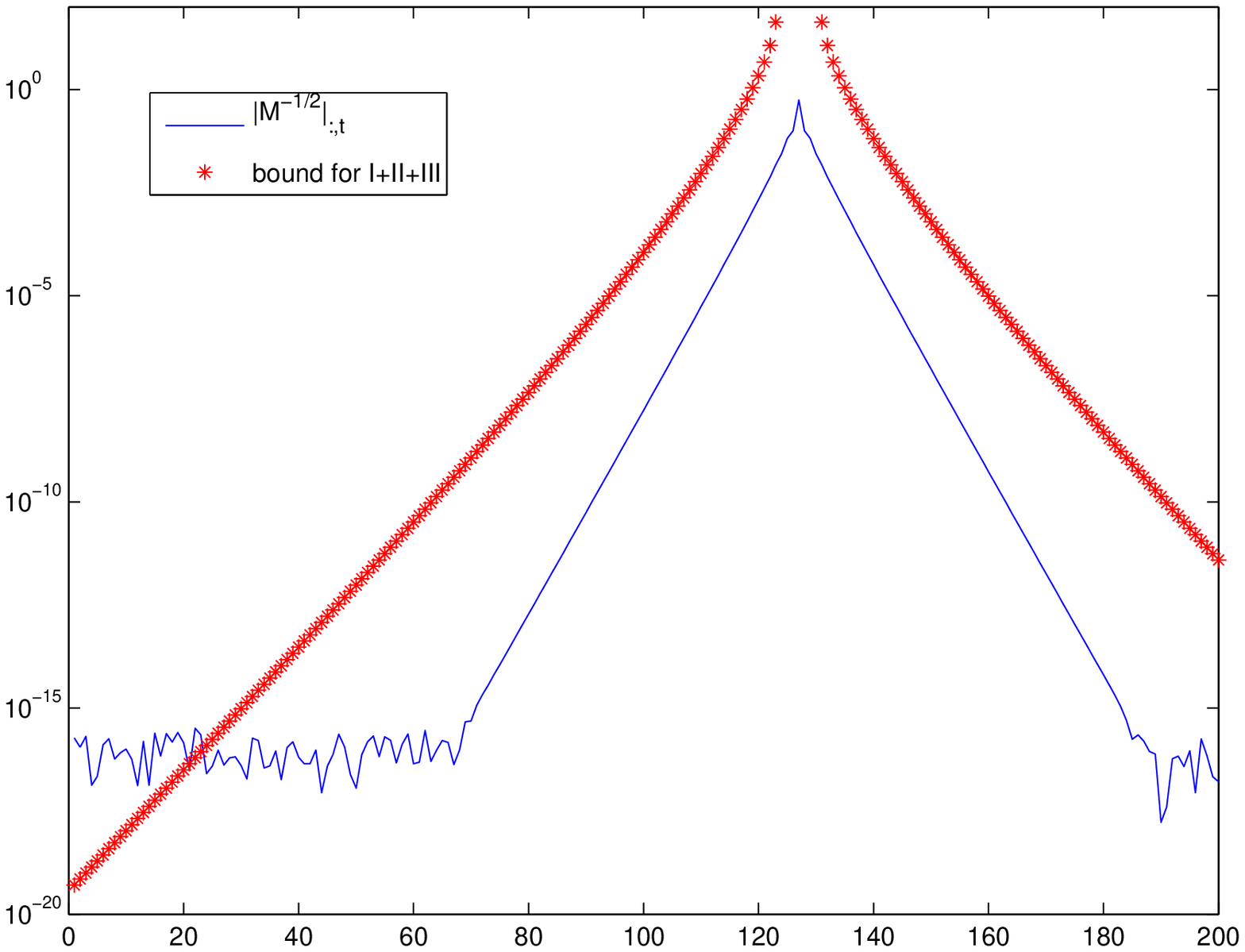}
\caption{Estimates for $|M^{-1/2}|_{:,t}$, $t=127$, using  I+II and the upper bound for III.
Size $n=200$, Log-scale.  Left: $M={\rm tridiag}(-1,4,-1)$.
Right: $M={\rm pentadiag}(-0.5,-1,4,-1,-0.5)$.
\label{fig:LS-1/2}}
\end{figure}

As yet another example, consider the entire function $f(x)=(1-\exp(-x))/x$ for $x\in [0,1]$,
which is a Laplace--Stieltjes function with ${\rm d}\alpha(\tau)={\rm d}\tau $ 
(see section \ref{sec:classes}).
Starting from (\ref{eqn:LSbound}) we can determine 
new terms $I, II$, and estimate $III$ as
it was done for the inverse square root. Due to the small interval size,
the first term $I$ accounts for the whole bound for most choices of $k,t$.
For the same two matrices used above,
the actual (superexponential) decay and its approximation are reported in Figure~\ref{fig:f3}.

\begin{figure}[thb]
\centering
\includegraphics[width=2.5in,height=2.5in]{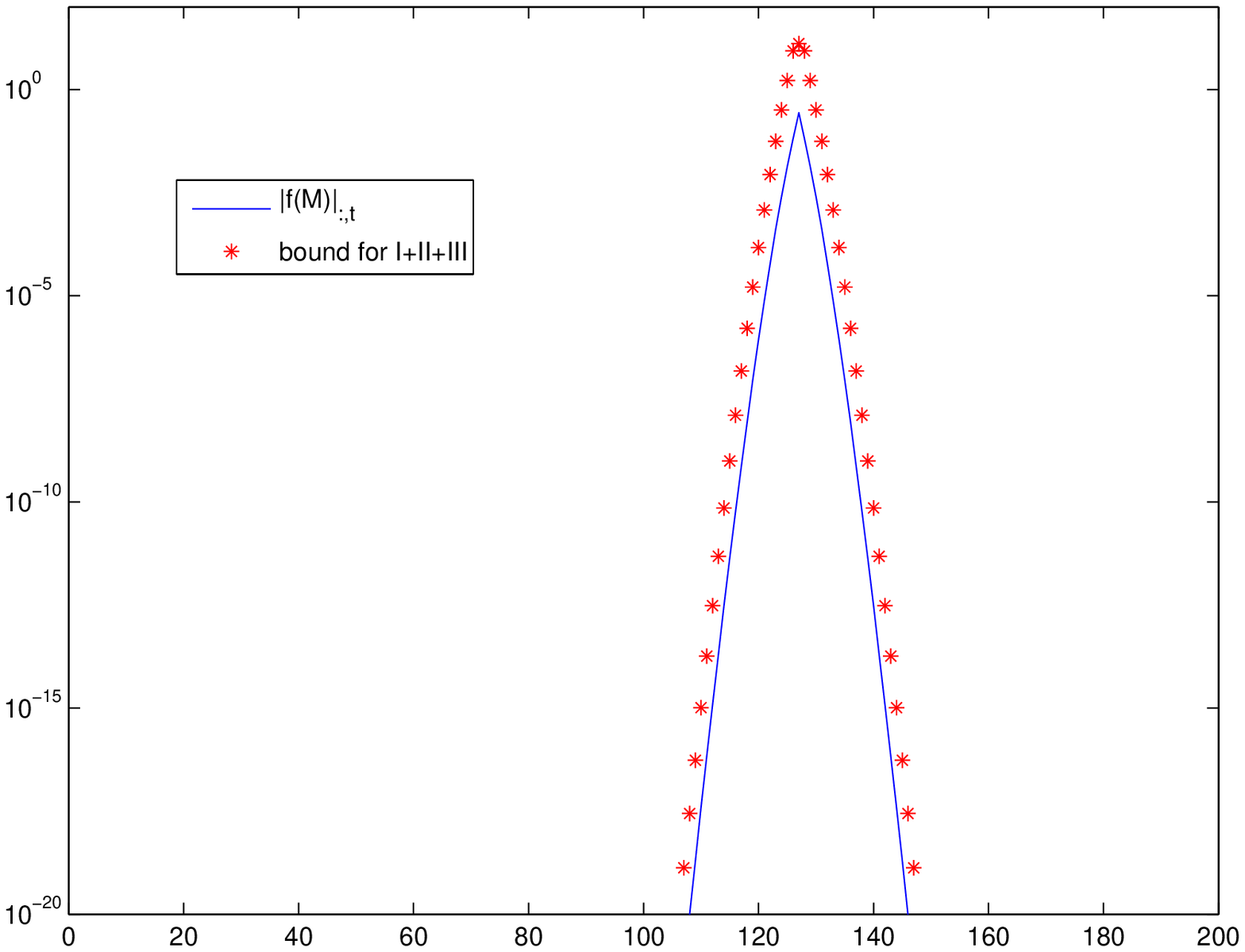}
\includegraphics[width=2.5in,height=2.5in]{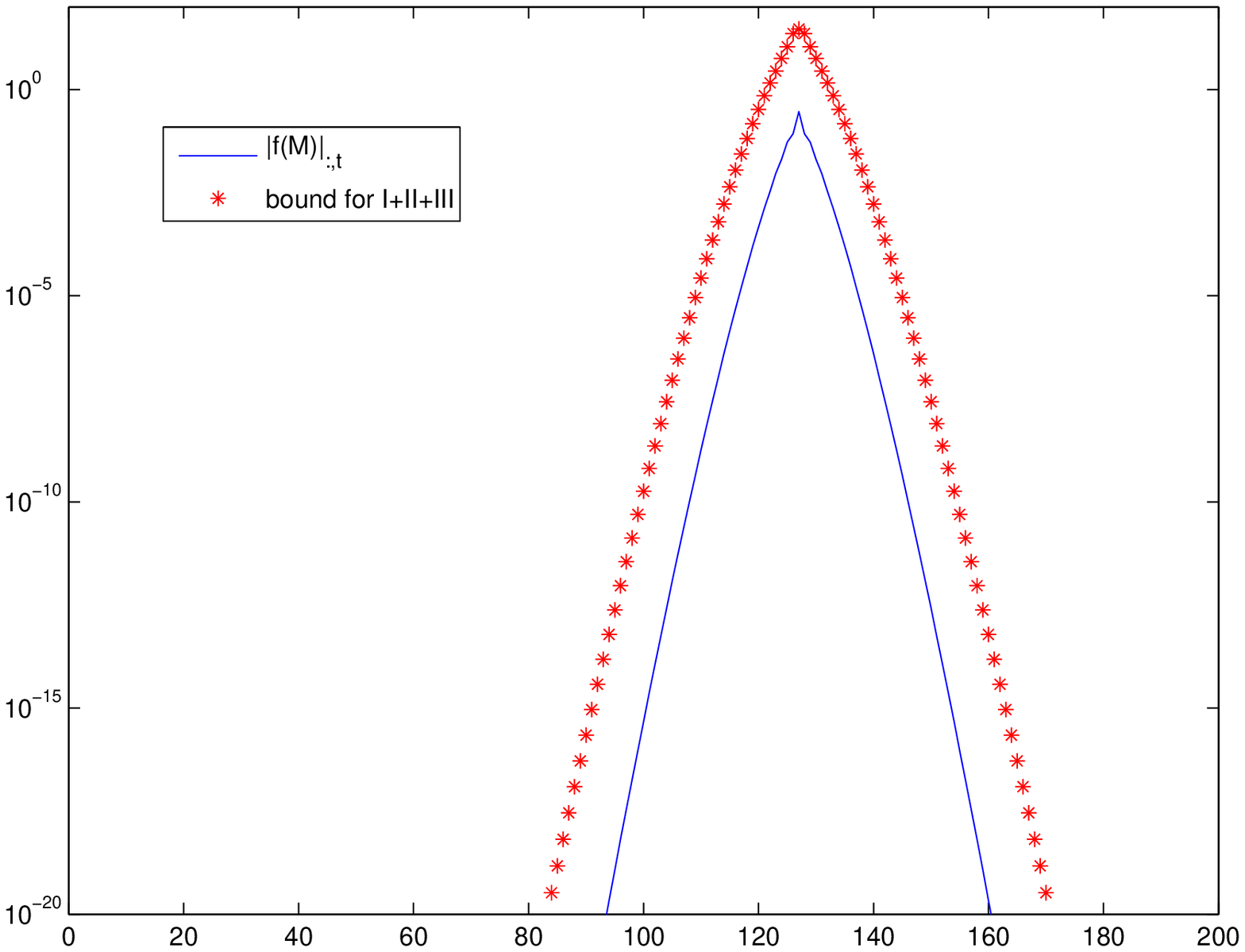}
\caption{Estimates for $|M^{-1}(I-\exp(-M))|_{:,t}$, $t=127$, using  I+II and the upper bound for III.
size $n=200$, Log-scale.  Left: $M={\rm tridiag}(-1,4,-1)$.
Right: $M={\rm pentadiag}(-0.5,-1,4,-1,-0.5)$.
\label{fig:f3}}
\end{figure}

We remark that for the validity of Theorem \ref{th:LS}, we cannot
relax the assumption that $M$ be positive definite. This makes sense
since we are considering functions of $M$ that are defined on
$(0,\infty)$. If $M$ is not positive definite but
$f$ happens to be defined on a larger interval containing the spectrum of $M$,
for instance on all of $\RR$, it may still be possible, in some cases,
to obtain bounds for $f(M)$ 
from the corresponding bounds on $f(M + \delta I)$
where the shifted matrix $M + \delta I$ is positive definite.

\vskip 0.05in
\begin{remark}\label{rem:shiftedM}
We observe that if $f(M+i \zeta I)$ is well defined for $\zeta\in\RR$, then the
estimate (\ref{eqn:LSbound}) also holds for
$|f(M+i \zeta I)|_{k,t}$, since $|\exp(i\zeta)|=1$.
\end{remark}

\subsection{Bounds for Cauchy--Stieltjes functions} 
Bounds for the entries of $f(M)$, where $f$ is a Cauchy--Stieltjes
function and $M=M^*$ is positive definite, can be obtained in a
similar manner, with the bound (\ref{eqn:demko}) of Demko et al.~\cite{DMS}
replacing the bounds on $\exp(-\tau M)$ from Theorem \ref{th:boundexp}.

%For a given $\omega \in\Gamma$, let
For a given $\omega \in \Gamma = (\infty, 0)$, let
$\kappa=\kappa(\omega)=(\lambda_{\max}-\omega)/(\lambda_{\min}-\omega)$, 
$q = q(\omega) = (\sqrt{\kappa}-1)/(\sqrt{\kappa}+1)$,
$C=C(-\omega)=\max\{1/(\lambda_{\min}-\omega), C_0\}$, with
$C_0 = C_0(-\omega) = (1+\sqrt{\kappa})^{1/2}/(2(\lambda_{\max}-\omega))$.
We immediately obtain the following result.

\begin{theorem}\label{th:CS}
Let $M=M^*$ be positive definite and let $f$ be a Cauchy--Stieltjes function. 
Then for all $k$ and $t$ we have
\begin{eqnarray}\label{eqn:Markov_general}
%|f(M)_{kt}|\le  \int_\Gamma  C q^{\frac{|k-t|}{\beta}} {\rm d}\gamma(\omega).
|f(M)_{kt}|\le  \int_{-\infty}^0  C(\omega) q(\omega)^{\frac{|k-t|}{\beta}} {\rm d}\omega.
\end{eqnarray}
\end{theorem}

For specific functions we can be more explicit, and provide more insightful upper bounds 
by evaluating or bounding the integral on the right-hand side of 
(\ref{eqn:Markov_general}).
As an example,
let us consider again $f(x) = x^{-\frac12}$, which happens to be both a Laplace--Stieltjes
and a Cauchy--Stieltjes function. In this case we find the bound
%\begin{theorem}
%$$
\begin{equation}\label{bound_sqrt}
|M_{kt}^{-\frac12}| \le
\frac 2 {\pi}
% \left ( \frac 1 {\lambda_{\max}} + \frac 1 {\lambda_{\min}} \right)
(C(0)+C_2)
\left( \frac{ \sqrt{\lambda_{\max}} - \sqrt{\lambda_{\min}}}{
\sqrt{\lambda_{\max}} + \sqrt{\lambda_{\min}}} \right )^{ \frac{|k-t|}{\beta}},
%\left ( \frac{\kappa(0) -1}{\kappa(0)+1}\right )^{\frac{|i-j|}{\beta}}
%$$
\end{equation}
where $C_2= \max\left\{  1, (1+\frac 1 2 \sqrt{\kappa(0)})^{\frac 1 2}\right\}$.
%\end{theorem}
%\begin{proof}
Indeed, for the given function and upon substituting $\tau=-\omega$, 
(\ref{eqn:Markov_general}) becomes
\begin{eqnarray}
|M_{kt}^{-\frac12}| &\le& \frac{1}{\pi} \int_0^\infty C(\tau) 
\left( \frac{ \sqrt{\lambda_{\max}+\tau} - \sqrt{\lambda_{\min}+\tau}}{%
\sqrt{\lambda_{\max}+\tau} + \sqrt{\lambda_{\min}+\tau}} 
\right )^{ \frac{|k-t|}{\beta}} \frac{1}{\sqrt{\tau}} {\rm d}\tau\\
&\le& \frac{1}{\pi}
\left( \frac{ \sqrt{\lambda_{\max}} - \sqrt{\lambda_{\min}}}{
\sqrt{\lambda_{\max}} + \sqrt{\lambda_{\min}}} \right )^{ \frac{|k-t|}{\beta}}
\int_0^\infty C(\tau) \frac{1}{\sqrt{\tau}} {\rm d}\tau.
\end{eqnarray}
Let $\phi(\tau)$ be the integrand function.
We split the integral as $\int_0^\infty \phi(\tau) {\rm d}\tau =
\int_0^1 \phi(\tau) {\rm d}\tau + \int_1^\infty \phi(\tau) {\rm d}\tau$. 
For the first integral,
we observe that $C(\tau) \le C(0)$, so that
\begin{eqnarray*}
\int_0^1 C(\tau) \frac{1}{\sqrt{\tau}} {\rm d}\tau \le
C(0) \int_0^1  \frac{1}{\sqrt{\tau}} {\rm d}\tau = 2 C(0).
\end{eqnarray*}
For the second integral, we observe that $C(\tau) \le C_2 \frac 1 {\tau}$ where
$C_2= \max\{  1, (1+\sqrt{\kappa(0)})^{1/2}/2\}$, so that
\begin{eqnarray*}
\int_1^\infty C(\tau) \frac{1}{\sqrt{\tau}} {\rm d}\tau \le
C_2 \int_1^\infty  \frac{1}{\tau\sqrt{\tau}} {\rm d}\tau = 2 C_2.
\end{eqnarray*}
Collecting all results the final upper bound (\ref{bound_sqrt}) follows.
%\end{proof}

We note that for this particular matrix function, using the approach
just presented results in much more explicit bounds than those obtained
earlier using the Laplace--Stieltjes representation, which required
the numerical evaluation of three integrals. Also, since the bound
(\ref{eqn:demko}) is known to be sharp (see \cite{DMS}), it is to be expected
that the bounds (\ref{bound_sqrt}) will be generally better
than those obtained in the previous
section. 
Figure~\ref{fig:CS-1/2} shows the accuracy of the bounds in
(\ref{bound_sqrt}) for the same matrices as in Figure~\ref{fig:LS-1/2},
where the Laplace--Stieltjes bounds were used. For both matrices,
the quality of the Cauchy--Stieltjes bound is clearly superior.

\begin{figure}[thb]
\centering
\includegraphics[width=2.5in,height=2.5in]{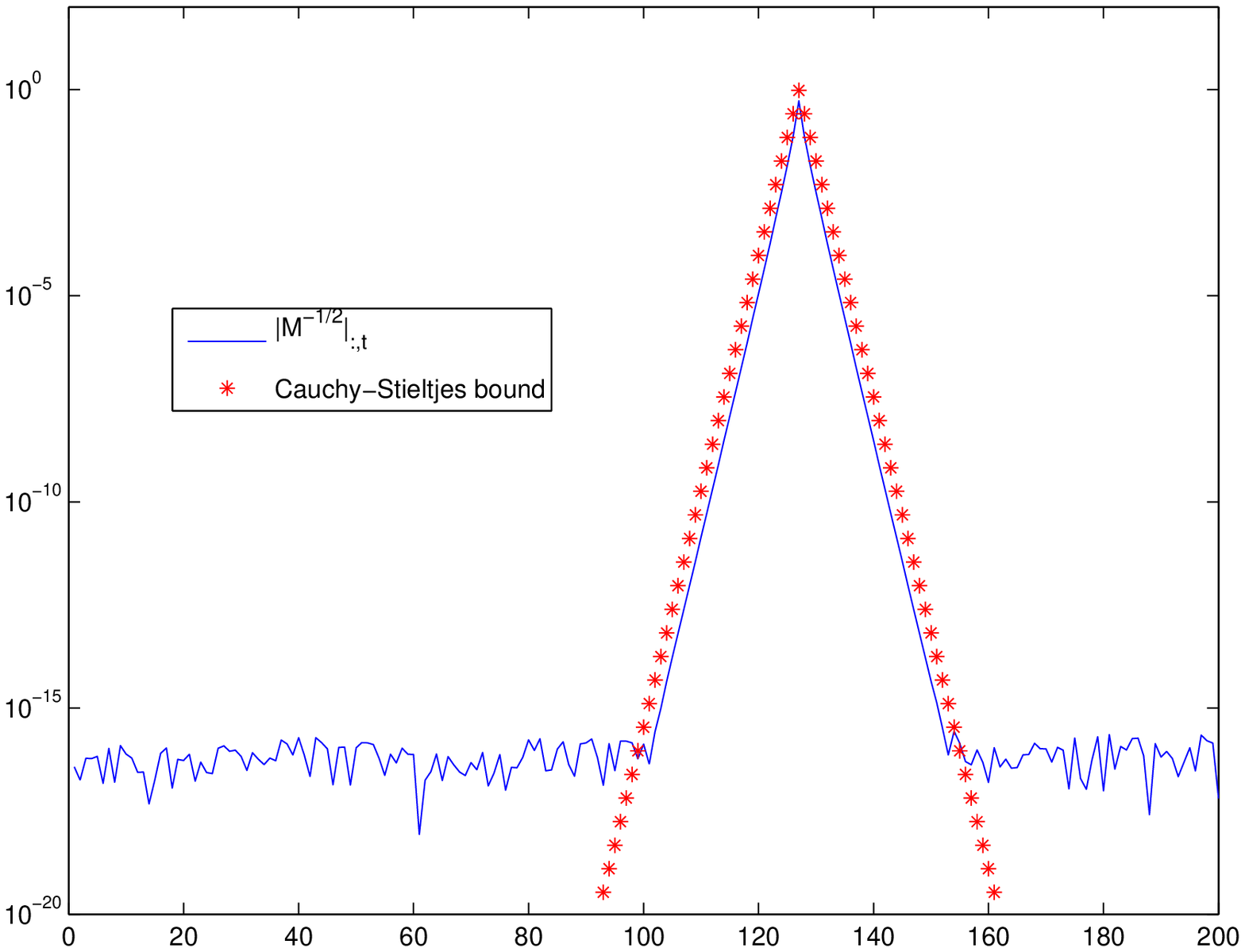}
\includegraphics[width=2.5in,height=2.5in]{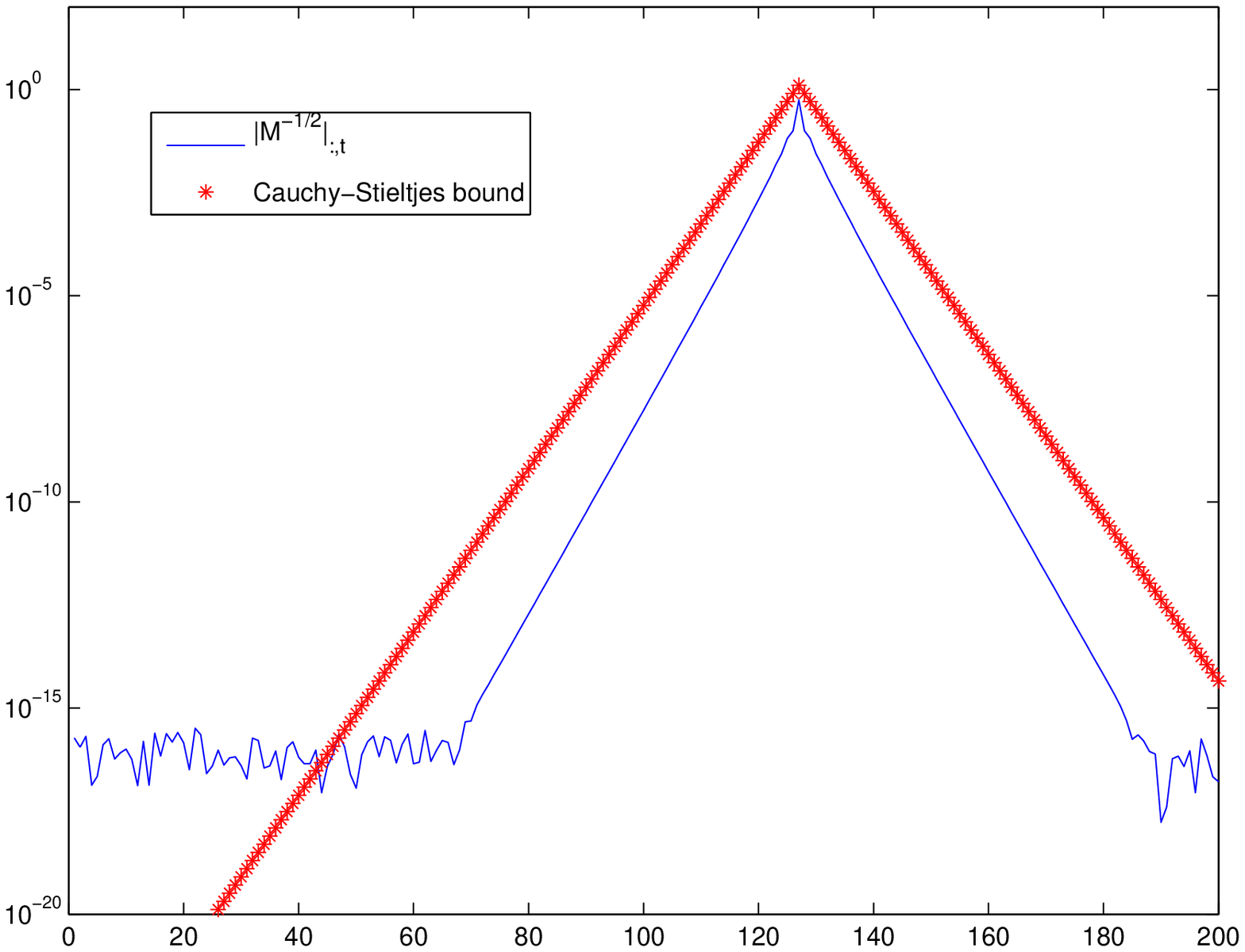}
\caption{Estimates for $|M^{-1/2}|_{:,t}$, $t=127$, using  (\ref{bound_sqrt}),
size $n=200$,  Log-scale.  Left: $M={\rm tridiag}(-1,4,-1)$.
Right: $M={\rm pentadiag}(-0.5,-1,4,-1,-0.5)$.
\label{fig:CS-1/2}}
\end{figure}

We conclude this section with a discussion on decay bounds for functions
of $M- i \zeta I$,
where $\zeta\in\RR$. These estimates may be useful when
the integral is over a complex curve.
We first recall a result of Freund for $(M-i \zeta I)^{-1}$. To this end,
we let again $\lambda_{\min}, \lambda_{\max}$ be the extreme eigenvalues of $M$ 
(assumed to be HPD),
and we let $\lambda_1 = \lambda_{\min} -  i\zeta$,
$\lambda_2 = \lambda_{\max} -  i \zeta$.
% and put in our
%framework in \cite[Proposition 2.1]{CanutoSimonciniVeraniLAA.14}. This estimate
%provides a bound for the quantity $| e_\ell^T(\omega I - M) e_i |$ above.

\begin{proposition}\label{prop:Freund}
Assume $M$ is Hermitian positive definite and $\beta$-banded. Let
%$a=(\lambda_1+\lambda_2)/(\lambda_2-\lambda_1)$, and 
$R>1$ be defined as $R=\alpha + \sqrt{\alpha^2-1}$,
with $\alpha=(|\lambda_1|+|\lambda_2|)/|\lambda_2-\lambda_1|$.
Then for $k\ne t$,
$$
|( M - i\zeta I )^{-1}|_{tk} \le 
%\frac{2R}{|\lambda_1-\lambda_2|} \frac{4R^2}{(R^2-1)^2}
C(\zeta)
\left (\frac{1}{R}\right )^{\frac{|t -k|}{\beta}}  \,\, {\rm with}\,\,
C(\zeta) = \frac{2R}{|\lambda_1-\lambda_2|} \frac{4R^2}{(R^2-1)^2} .
$$
\end{proposition}

With this bound, we can modify (\ref{eqn:Markov_general}) so as to
handle more general matrices as follows.
Once again, we let $\lambda_{\min}, \lambda_{\max}$ be the extreme eigenvalues of $M$,
and now we let $\lambda_1 = \lambda_{\min} -  i\zeta - \omega$,
$\lambda_2 = \lambda_{\max} -  i \zeta - \omega$; $\alpha$ and $R$ are defined 
accordingly.
\begin{eqnarray}\label{eqn:shiftedM_CS}
|f(M-i\zeta I)|_{kt} \le \int_{-\infty}^0 C %(\zeta,\omega) 
\left (\frac{1}{R}\right )^{\frac{|k -t|}{\beta}} {\rm d}\gamma(\omega), \quad k\ne t .
\quad
\end{eqnarray}
Since $R=R(\zeta,\omega)$ is defined in terms of 
spectral information of the shifted matrix $M-i\zeta I - \omega I$, 
we also obtain
$C=C(\zeta,\omega) = 
\frac{2R(\zeta,\omega)}{|\lambda_{\max}-\lambda_{\min}|} \frac{4R(\zeta,\omega)^2}{(R(\zeta,\omega)^2-1)^2}$.

%\begin{remark}
%Add remark on $f(M+i z I)$, where Freund's bound replaces the
%integrand in (\ref{eqn:Markov_general}).
%\end{remark}

\section{Extensions to more general sparse matrices}  \label{sec:ext}
Although all our main results so far have been stated for matrices that
are banded, it is possible to extend the previous bounds to functions of matrices
with general sparsity patterns.

Following the approach in \cite{CEPD06} and \cite{Benzi2007}, 
let $G=(V,E)$ be the undirected graph
describing the nonzero pattern of $M$. Here $V$ is a set of $n$ vertices
(one for each row/column of $M$) and $E$ is a set of edges. The set $E\subseteq
V\times V$ is defined as follows: there is
an edge $(i,j)\in E$ if and only if $M_{ij}\ne 0$ (equivalently, $M_{ji}\ne 0$
since $M=M^*$). Given any two nodes $i$
and $j$ in $V$, a {\em path of length $k$} between $i$ and $j$ is a sequence of nodes
$i_0=i,i_1,i_2,\ldots ,i_{k-1},i_k=j$ such that $(i_{\ell},i_{\ell +1})\in E$
for all $\ell = 0,1,\ldots ,k-1$ and $i_{\ell}\ne i_m$ for $\ell \ne m$.
If $G$ is connected (equivalently, if $M$ is irreducible, which we will
assume to be the case), then there exists a path between any two nodes
$i,j\in V$. The {\em geodesic distance} $d(i,j)$ between two nodes
$i,j\in G$ is then the length of
the shortest path joining $i$ and $j$. With this distance, $(G,d)$ is a metric 
space. 

We can then extend every one of the bounds seen so far for banded $M$
to a general sparse matrix $M=M^*$ simply
by systematically
replacing the quantity $\frac{|k-t|}{\beta}$ by the geodesic distance $d(k,t)$. 
Hence, the decay in the entries of $f(M)$ is to be understood in terms of
distance from the nonzero pattern of $M$, rather than away from the main
diagonal. 
We refer again to \cite{Benzi2007} for details.
We note that this extension easily carries over to the bounds presented in the following
section.

Finally, we observe that all the results in this paper apply
to the case where $M$ is an infinite matrix with bounded spectrum, 
provided that $f$ has no singularities
on an open neighborhood of the spectral interval $[\lambda_{\min}, \lambda_{\max}]$.
This implies that our bounds apply to all the $n\times n$ principal
submatrices (``finite sections") of such matrices, and that the bounds
are uniform in $n$ as $n\to \infty$. 

\section{Estimates for functions of Kronecker sums of matrices}\label{sec:Kron}
The decay pattern for matrices with Kronecker structure has a rich
structure. In addition to a decay away from the diagonal, which
depends on the matrix bandwidth, a ``local'' decay can
be observed within the bandwidth; see Figure \ref{fig:DecKron}. This particular pattern was
described for $f(x)=x^{-1}$ in \cite{CanutoSimonciniVeraniLAA.14}; here
we largely expand on the class of functions for which the phenomenon
can be described. 

\begin{figure}[thb]
\centering
\includegraphics[width=2.5in,height=2.5in]{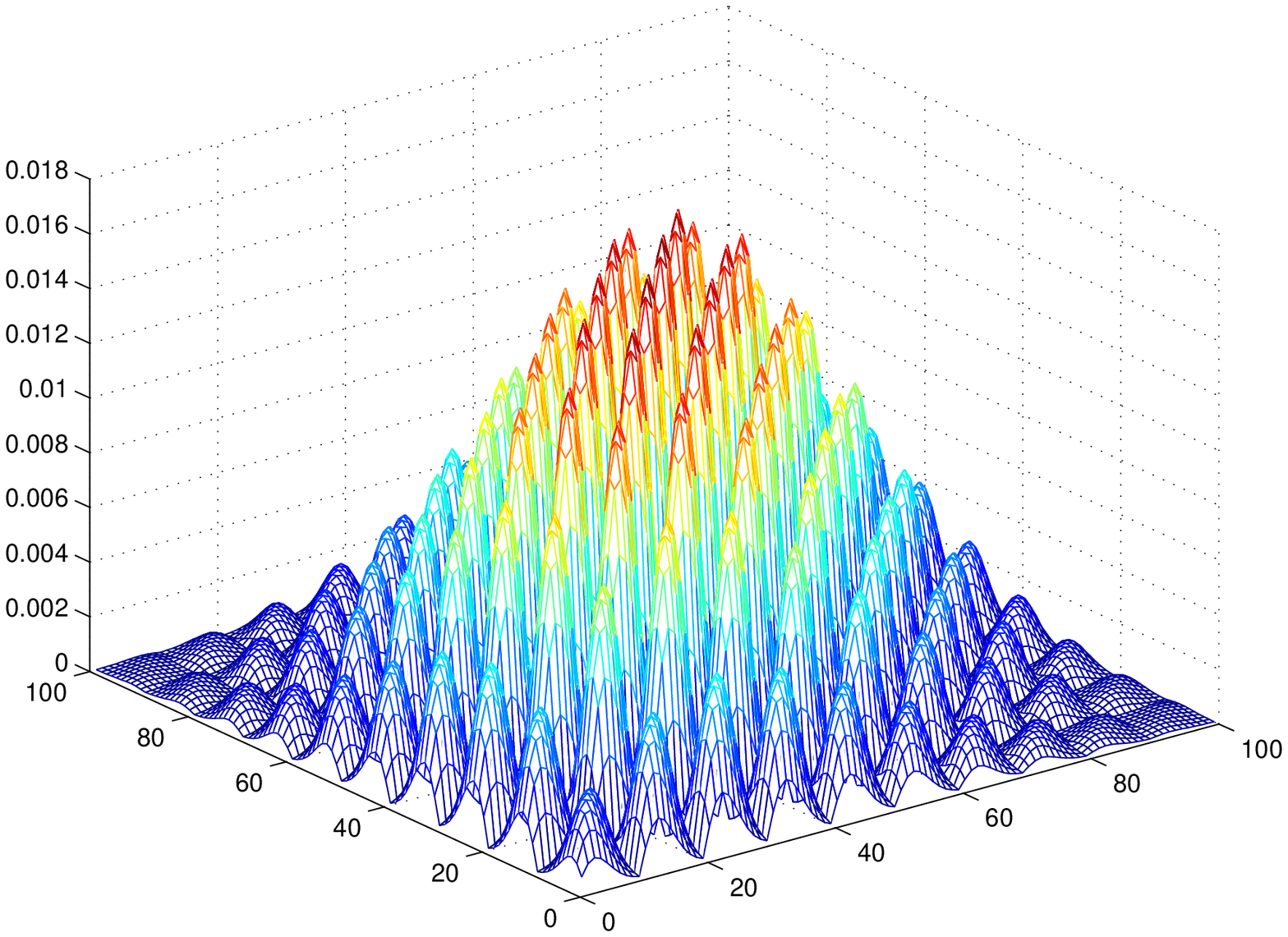}
\includegraphics[width=2.5in,height=2.5in]{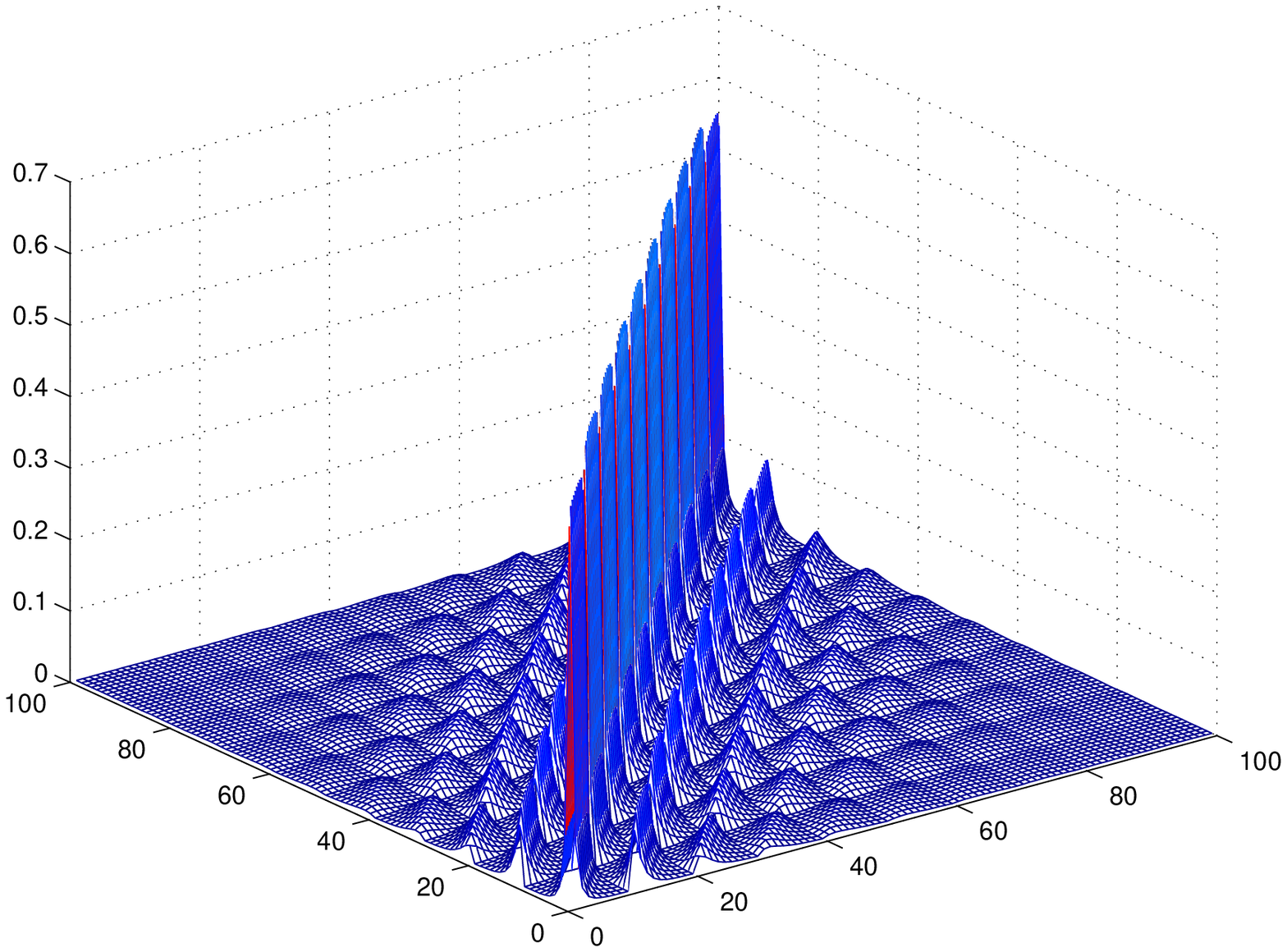}
\caption{%Example \ref{ex:kron_f3}.
Three-dimensional decay plots for $[f({\cal A})]_{ij}$
where $\cal A$ is the 5-point finite difference 
discretization of the negative Laplacian on the unit square on a 
$10\times 10$ uniform grid with zero Dirichlet boundary conditions.
Left: $f({\cal A}) = \exp(-5 {\cal A})$.
Right: $f({\cal A}) = {\cal A}^{-1/2}$. 
\label{fig:DecKron}}
\end{figure}

\vskip 0.03in

Some matrix functions enjoy properties that make their
application to Kronecker sums of matrices particularly simple. This is the
case for instance  of the exponential and certain trigonometric
functions like $\sin(x)$ and $\cos(x)$.
For these, bounds for their entries can be directly obtained from
the estimates of the previous sections.

%%%%%%%%%%%%%%%%%%%%%%%%%%%%%%%%%%%%%%%%%%%%%%%
\subsection{The exponential function} % of Kronecker sums of matrices}
%, that
%allow us to considerably simplify the derivation of a sharp bound
%for the entries of the exponential matrix applied to the matrix ${\cal A}$ in  
%(\ref{eqn:kron}).
Recall the relation (\ref{eqn:exp_kron}),
%which allows us to considerably simplify the derivation of a sharp bound
%for the entries of the exponential matrix applied to the matrix ${\cal A}$ in  
%(\ref{eqn:kron}),
%in Kronecker form.
which implies that
\begin{equation} \label{eqn:kron_tau}
\exp(-\tau {\cal A}) = \exp(-\tau M)\otimes  \exp(-\tau M), \quad \tau\in \RR 
\end{equation}
when ${\cal A} = M\otimes I + I \otimes M$.
Here and in the following, a lexicographic ordering of
the entries will be used, so that each row or column index $k$ of
${\cal A}$
corresponds to the pair $k=(k_1,k_2)$ in the two-dimensional Cartesian grid.
Furthermore, for any fixed values of $\tau, \rho, \beta >0$, define
\begin{equation}\label{eqn:Phi}
\Phi (i,j) := 
\left\{\begin{array}{ll} 10\exp \left (-\frac{(|i-j|/\beta)^2}{5\rho \tau}\right ), 
& \textnormal{for}\quad \sqrt{4\rho\tau} \le \frac{|i-j|}{\beta} \le 2\rho \tau,\\ 
10\frac{\exp(-\rho\tau)}{\rho \tau} \left ( \frac{{\rm e}\rho\tau}{\frac{|i-j|}{\beta}}
\right )^{\frac{|i-j|}{\beta}} , &
\textnormal{for}\quad \frac{|i-j|}{\beta} \ge 2\rho\tau .
\end{array}\right .
\end{equation}

Note that $\Phi(i,j)$ is only defined for $|i-j| > \sqrt{4\rho \tau} \beta$.
With these notations, the following bounds can be obtained. 

\begin{theorem}\label{th:exp}
Let $M$ be Hermitian and positive semidefinite with bandwidth $\beta$ 
 and spectrum contained in $[0,4\rho]$,
and let ${\cal A}=I\otimes M + M \otimes I$.
Then
$$
(\exp(-\tau {\cal A}))_{kt} = (\exp(-\tau M))_{k_1 t_1} (\exp(-\tau M))_{k_2 t_2} .
%(\exp(-{\cal A}))_{kt} = (\exp(-M))_{\ell i} (\exp(-M))_{jm} ,
$$
%where $j=\lfloor (t-1)/n\rfloor+1$, $i=t-n \lfloor (t-1)/n\rfloor$ and
%$m=\lfloor (k-1)/n\rfloor+1$, $\ell=k-n \lfloor (k-1)/n\rfloor$.
Therefore, % with the notation introduced above,
for $\tau > 0$ 
%$$
%|(\exp(-{\cal A}))_{kt}| \le K^2 \left(\frac 1 R\right)^{\frac {|t_1-k_1| 
%+ |t_2-k_2|}{\beta}}
%$$
%where $K=K(1)$ and $R$ are as defined in (\ref{constants}).
$$|(\exp(-\tau {\cal A}))_{kt}| \le \Phi(k_1,t_1) \Phi (k_2,t_2)$$
for all $t=(t_1,t_2)$ and $k=(k_1,k_2)$ with
$\min \{|t_1-k_1|, |t_2-k_2|\}\ge \sqrt{4\rho\tau} \beta$.
\end{theorem}

\begin{proof}
Using (\ref{eqn:exp_kron}) we obtain
$$
e_k^T \exp(-{\tau \cal A}) e_t = e_k^T \exp(-\tau M)\otimes \exp(-\tau M) e_t.
$$
Let $E_{t_1 t_2}$ be the $n\times n$ matrix such that $e_t={\rm vec}(E_{t_1 t_2})
\in\RR^{n^2}$, and
in particular $E_{t_1 t_2} = e_{t_1} e_{t_2}^T$, with $e_{t_1}, e_{t_2} \in\RR^n$. Then
\begin{eqnarray*}
e_k^T \exp(-\tau M)\otimes \exp(-\tau M) e_t &=& 
e_k^T {\rm vec}( \exp(-\tau M) E_{t_1 t_2} \exp(-\tau M)^*) \\
&=&e_k^T {\rm vec}( \exp(-\tau M) e_{t_1} e_{t_2}^T \exp(-\tau M)^*) \\
&=& 
e_k^T 
\begin{bmatrix}
\exp(-\tau M) e_{t_1} (e_{t_2}^T \exp(-\tau M)^*)e_1 \\
\exp(-\tau M) e_{t_1} (e_{t_2}^T \exp(-\tau M)^*)e_2 \\
\vdots \\
\exp(-\tau M) e_{t_1} (e_{t_2}^T \exp(-\tau M)^*)e_n 
\end{bmatrix} \\
&=& e_{k_1}^T \exp(-\tau M) e_{t_1} (e_{t_2}^T \exp(-\tau M)^*)e_{k_2}) ,
\end{eqnarray*}
which proves the first relation for $M$ Hermitian. % symmetric. 
For the bound, it is sufficient to use (\ref{eqn:Phi}) to obtain
%$$
%|e_k^T \exp(-\tau {\cal A}) e_t | = |e_{k_1}^T \exp(-\tau M) e_{t_1}| \, 
%|e_{t_2}^T \exp(-\tau M)^T)e_{k_2}| \le
%K^2 \!\! \left(\frac 1 R\right )^{\frac{|{t_1}-k_1|}{\beta}}
%\!\!\!\! \left(\frac 1 R\right )^{\frac{|t_2-k_2|}{\beta}},
%% = K^2 q^{\frac{|i-\ell|+|j-m|}{b}} .
%$$
%from which the result follows.
the desired conclusion.
\end{proof}

The result can be easily generalized to a broader class of matrices.

\begin{corollary}
Let ${\cal A}=I\otimes M_1 + M_2 \otimes I$ with $M_1$ and $M_2$ having bandwidth
$\beta_1$ and $\beta_2$, respectively. Also, let the spectrum of $M_1$ be
contained in the interval $[0,4\rho_1]$ and that of $M_2$ in the
interval $[0,4\rho_2]$, with $\rho_1,\rho_2 > 0$.  Then 
for $t=(t_1,t_2)$ and $k=(k_1,k_2)$,
with
$|t_\ell-k_\ell|\ge \sqrt{4\rho_\ell\tau} \beta_\ell$, $\ell=1,2$,
$$
(\exp(-\tau {\cal A}))_{kt} = (\exp(-\tau M_1))_{k_1 t_1} (\exp(-\tau M_2))_{t_2 k_2} .
%(\exp(-{\cal A}))_{kt} = (\exp(-M_1))_{\ell i} (\exp(-M_2))_{jm} ,
$$
%where $j=\lfloor (t-1)/n\rfloor+1$, $i=t-n \lfloor (t-1)/n\rfloor$ and
%$m=\lfloor (k-1)/n\rfloor+1$, $\ell=k-n \lfloor (k-1)/n\rfloor$. 
Therefore, % {\color{red}for $\tau\rho_1,\tau\rho_2 > 1$}
$$
|(\exp(-\tau {\cal A}))_{kt}| \le 
%K_1K_2 \left (\frac 1 {R_1}\right)^{\frac {|t_1-k_1|}{\beta_1}}
% \left (\frac 1 {R_2}\right)^{\frac {|t_2-k_2|}{\beta_2}} ,
\Phi_1 (k_1,t_1) \Phi_2(k_2,t_2)
$$
where $\Phi_k (i,j)$ is defined as $\Phi (i,j)$ in (\ref{eqn:Phi}) with
$\rho_\ell$, $\beta_\ell$ replacing $\rho$, $\beta$.
%\equiv K_i(1), R_i$ are as defined in (\ref{constants}), for $i=1,2$.
\end{corollary}

Generalization to the case of Kronecker sums of more
than two matrices is relatively straightforward.
Consider for example the case of three summands.
A lexicographic order of
the entries is again used, so that each row or column index $k$ of
${\cal A} = M\otimes I\otimes I + I\otimes M\otimes I + I\otimes I\otimes M$ 
corresponds to a triplet $k=(k_1,k_2,k_3)$ in the three-dimensional Cartesian grid.

\begin{corollary}
Let $M$ be $\beta$-banded, Hermitian and with spectrum contained 
in $[0,4\rho]$, and
let ${\cal A} = 
M\otimes I\otimes I +
I\otimes M\otimes I +
I\otimes I\otimes M$ and $k = (k_1, k_2, k_3)$ and $t=(t_1,t_2,t_3)$. 
Then
$$
(\exp(-\tau {\cal A}))_{kt} =
(\exp(-\tau M))_{k_1,t_1} (\exp(-\tau M))_{t_2,k_2}  (\exp(-\tau M))_{t_3,k_3}  ,
$$
from which it follows
%from which {\color{red} for $\tau\rho>1$} it follows
$$
|(\exp(-\tau {\cal A}) )_{kt}| %\le K^3
%\lambda^{\frac{|k_1-t_1|+|k_2-t_2|+|k_3-t_3|}{\beta}} ,
\le \Phi(k_1,t_1)\Phi(k_2,t_2)\Phi(k_3,t_3),
$$
%where $K\equiv K(1)$ and $R$ are as defined in (\ref{constants}).
for all $(k_1,t_1), (k_2,t_2), (k_3,t_3)$
with $\min \{|k_1-t_1|,|k_2-t_2|,|k_3-t_3|\} > \sqrt{4\tau \rho}\beta$.
\end{corollary}

\begin{proof}
We write
${\cal A} = 
%M\otimes (I_n\otimes I_n) +
%I_n\otimes (M\otimes I_n + I_n\otimes M )$, % =: M \otimes I + I \otimes M_2 .
M\otimes (I\otimes I) +
I\otimes (M\otimes I + I\otimes M )$, % =: M \otimes I + I \otimes M_2 .
so that
\begin{eqnarray*}
\exp(-\tau {\cal A}) &=& \exp(-\tau M) \otimes \exp(-\tau M\otimes I + I\otimes (-\tau M) ) \\
&=&
\exp(-\tau M) \otimes \exp(-\tau M) \otimes \exp(-\tau M).
\end{eqnarray*}
 
Therefore, using $i=(t_1,t_2)$,
\begin{eqnarray*}
(\exp(-\tau {\cal A}))_{kt} &=&
e_k^T {\rm vec}
\left ( (\exp(-\tau M)\otimes \exp(-\tau M))e_{i} e_{t_3}^T \exp(-\tau M)  \right ) \\
&=& e_k^T {\rm vec}\left( {\rm vec}(\exp(-\tau M) e_{t_1} e_{t_2}^T \exp(-\tau M) ) 
e_{t_3}^T \exp(-\tau M) \right ) \\
&=& e_k^T {\rm vec}
\left ( 
\begin{bmatrix}
\exp(-\tau M) e_{t_1} (e_{t_2}^T \exp(-\tau M)e_1) \\
\exp(-\tau M) e_{t_1} (e_{t_2}^T \exp(-\tau M)e_2) \\
\vdots \\
\exp(-\tau M) e_{t_1} (e_{t_2}^T \exp(-\tau M)e_n) 
\end{bmatrix}
\right ) e_{t_3}^T \exp(-\tau M) \\
&=& (e_{k_1}^T \exp(-\tau M) e_{t_1}) \,(e_{t_2}^T \exp(-\tau M)e_{k_2})\, (e_{t_3}^T \exp(-\tau M)e_{k_3}) .
\end{eqnarray*}
The rest follows as in the proof of Theorem \ref{th:exp}.
\end{proof}

%%%%%%%%%%%%%%%%%%%%%%%%%%%%%%%%%%%%
%\subsection{Bounds for the matrix sine and cosine}
\begin{remark}
Using (\ref{eqn:sin_kron}),
one can obtain similar bounds for $\cos({\cal A})$ and $\sin({\cal A})$,
where ${\cal A} = M_1\otimes I + I \otimes M_2$ with $M_1$, $M_2$ banded.
\end{remark}

%%%%%%%%%%%%%%%%%%%%%%%%%%%%%%%%%%%%
%\section{Decay estimates for functions of Kronecker sums}\label{sec:Kron}
%Consider again ${\cal A} = M\otimes I + I \otimes M$ with $M=M^*$ having bandwidth $\beta$
%and positive definite.
%
%%%%%%%%%%%%%%%%%%%%%%%%%%%%%%%%%%
\subsection{Laplace--Stieltjes functions}
If $f$ is a Laplace--Stieltjes function, then
%$f(M)$ is well-defined, and
%$$
%f(M) = \int_{0}^\infty \exp (-\tau M) {\rm d}\alpha(\tau).
%$$
%Similarly, 
$f({\cal A})$ is well-defined and exploiting the relation 
(\ref{eqn:exp_kron}) we can write
$$
f({\cal A}) = \int_{0}^\infty \exp (-\tau{\cal A}) {\rm d}\alpha(\tau)
              = \int_{0}^\infty \exp (-\tau M)\otimes \exp (-\tau M) {\rm d}\alpha(\tau).
$$

Thus, using $k=(k_1,k_2)$ and $t=(t_1,t_2)$,
\begin{eqnarray}%\label{basic_id}
(f({\cal A}))_{kt} & = & \int_{0}^\infty e_k^T\exp (-\tau M)\otimes \exp (-\tau M)e_t 
{\rm d}\alpha(\tau)\nonumber\\
&=& \int_{0}^\infty (\exp (-\tau M))_{k_1 t_1}(\exp (-\tau M))_{t_2 k_2}{\rm d}\alpha(\tau) .
\nonumber
\end{eqnarray}

%It follows that
%\begin{equation}\label{eqn:int_bd}
%|f({\cal A})_{kt}\le \int_0^{\infty} \Phi(k_1,t_1) \Phi(k_2,t_2){\rm d}\alpha(\tau) ,
%\end{equation}
%for all $t=(t_1,t_2)$ and $k=(k_1,k_2)$ with
%$\min \{|t_1-k_1|, |t_2-k_2|\}\ge \sqrt{4\rho\tau} \beta$.

With the notation of Theorem \ref{th:LS}, we have
\begin{equation}\label{eqn:LS_kron}
|f({\cal A})|_{kt}\le \int_0^{\infty} \exp(-2\lambda_{\min} \tau) 
|\exp(-\tau \widehat M)|_{k_1 t_1} |\exp(-\tau \widehat M)|_{k_2 t_2} {\rm d}\alpha(\tau) .
\end{equation}
In this form, the bound (\ref{eqn:LS_kron}), of course, is not particularly 
useful. Explicit bounds can be obtained, for specific examples of 
Laplace--Stieltjes functions, by evaluating or bounding the integral
on the right-hand side of (\ref{eqn:LS_kron}). 

For instance,
using once again the inverse square root, so that $\alpha(\tau) = \sqrt{\pi/\tau}$,
we obtain
\begin{eqnarray}
|{\cal A}^{-\frac 1 2}|_{kt}&\le&
\sqrt{\pi} \int_0^{\infty} \frac {1}{\sqrt{\tau^3}} \exp(-2\lambda_{\min} \tau) 
|\exp(-\tau \widehat M)|_{k_1 t_1} |\exp(-\tau \widehat M)|_{k_2 t_2} {\rm d}\tau 
\label{eqn:LS-1/2kron}\\
&\le &
\sqrt{\pi} 
\left(\int_0^{\infty} \left (\frac 1 {\tau^{3/4}} \exp(-\lambda_{\min} \tau) 
|\exp(-\tau \widehat M)|_{k_1 t_1}\right)^2 {\rm d}\tau\right)^{\frac 1 2} \cdot \nonumber\\
& & \qquad
\left(\int_0^{\infty} \left (\frac 1 {\tau^{3/4}} \exp(-\lambda_{\min} \tau) 
|\exp(-\tau \widehat M)|_{k_2 t_2}\right)^2 {\rm d}\tau\right)^{\frac 1 2} \nonumber .
\end{eqnarray}
The two integrals can then be bounded as done in 
Theorem~\ref{th:LS}.
For the other example we have considered earlier,
namely the function $f(x)=(1-\exp(-x))/x$, the bound is the same except that
$\sqrt{\pi}/\sqrt{\tau^3}$ is replaced by one, and the integration interval 
reduces to $[0,1]$; see
also Example~\ref{ex:kron_f3} next.
\vskip 0.03in

\begin{example}\label{ex:kron_f3}
{\rm
We consider again the function $f(x) = (1-\exp(-x))/x$, and 
the two choices of matrix $M$ in Example~\ref{ex:exp}; 
for each of them we build $\cal A$ as the Kronecker sum ${\cal A}=M\otimes I + I\otimes M$.
The entries of the $t$th column with $t=94$, that is $(t_1,t_2)=(14,5)$ are
shown in Figure~\ref{fig:kronLS_f3}, together with the bound obtained above.
The oscillating pattern is well captured in both cases, with a particularly
good accuracy also in terms of magnitude in the tridiagonal case.
The lack of approximation near the diagonal reflects the condition 
$|k_i - t_i|/\beta\ge 2$, $i=1,2$.
}
\end{example}

\begin{figure}[thb]
\centering
\includegraphics[width=2.5in,height=2.5in]{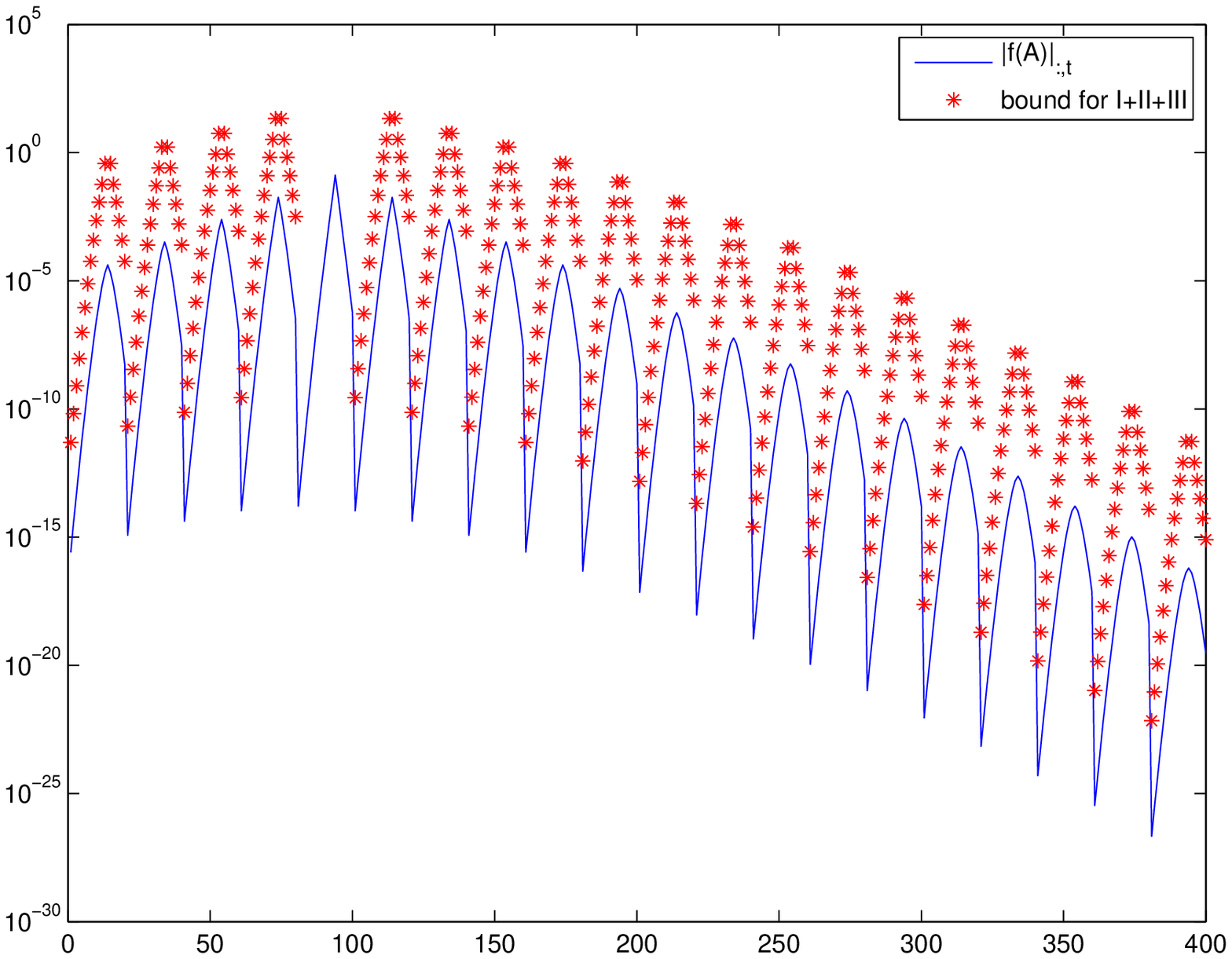}
\includegraphics[width=2.5in,height=2.5in]{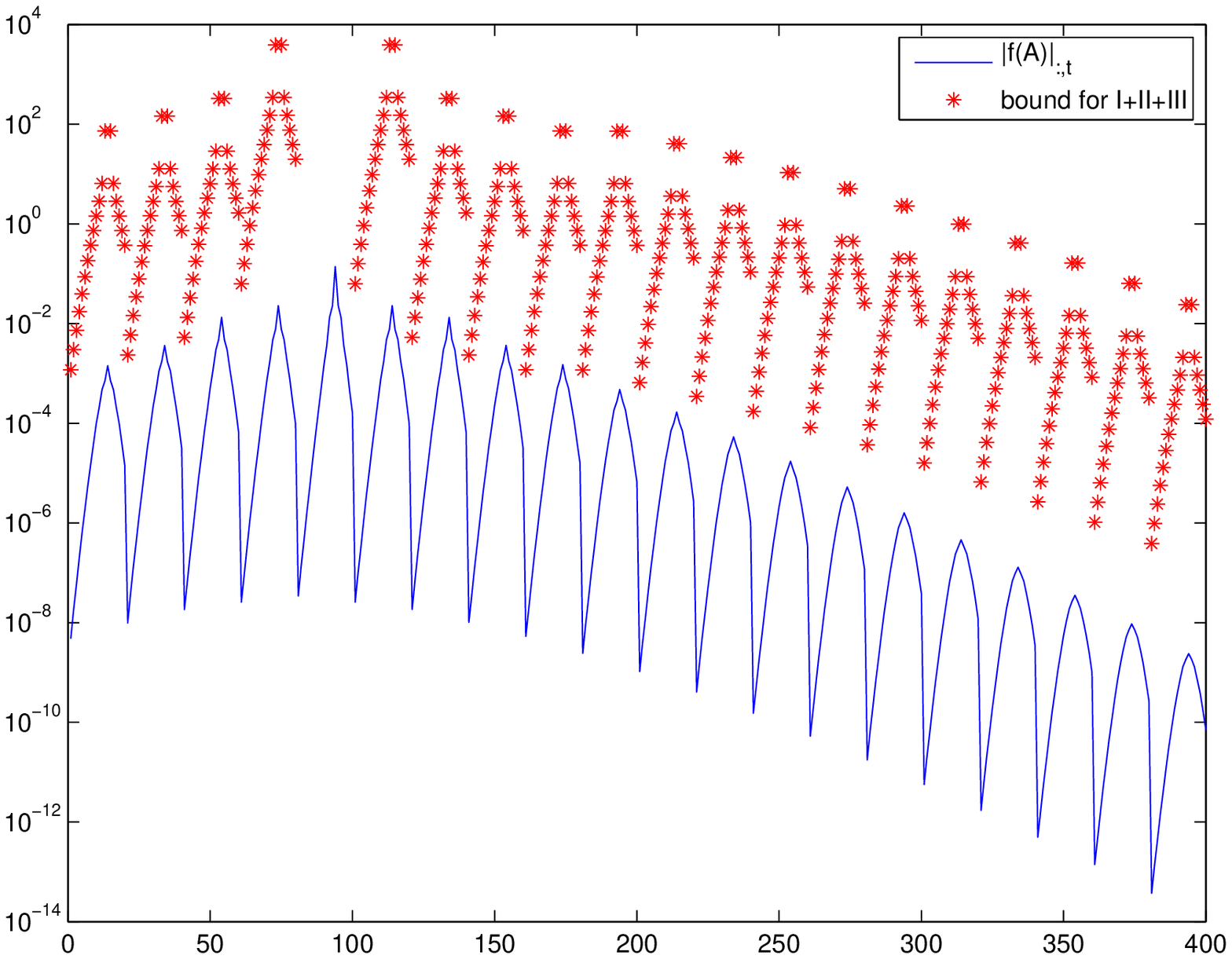}
\caption{Example \ref{ex:kron_f3}.
True decay and estimates for $|f(A)|_{:,t}$, $t=94$,
$A=M\otimes I +I \otimes M$ of 
size $n=400$.  Left: $M={\rm tridiag}(-1,4,-1)$.
Right: $M={\rm pentadiag}(-0.5,-1,4,-1,-0.5)$.
\label{fig:kronLS_f3}}
\end{figure}

\vskip 0.03in

\begin{remark}
These results can be easily extended to the case where
${\cal A} =  M_1\otimes I + I \otimes M_2$ with $M_1$, $M_2$ both
Hermitian positive definite and having bandwidths $\beta_1$, $\beta_2$.
It  can also be generalized to the case where ${\cal A}$ is the Kronecker
sum of three or more banded matrices.
\end{remark}

%%%%%%%%%%%%%%%%%%%%%%%%%%%%%%%%%%%%%%%%%%%%%
\subsection{Cauchy--Stieltjes functions}
If $f$ is a Cauchy--Stieltjes function and $\cal A$ has no eigenvalues on the
closed set $\Gamma \subset \CC$, then 
$$
f({\cal A}) = 
\int_\Gamma ({\cal A} -\omega I)^{-1} {\rm d}\gamma (\omega) ,
$$
so that
$$
e_k^Tf({\cal A}) e_t = 
\int_\Gamma e_k^T({\cal A} - \omega I)^{-1}e_t {\rm d}\gamma (\omega) .
$$

We can write
${\cal A} - \omega I = M\otimes I + I \otimes ( M - \omega I)$. 
Each column $t$ of the matrix inverse, $x_t := (\omega I - {\cal A})^{-1}e_t$, may be viewed as
the matrix solution $X_t=X_t(\omega)\in\CC^{n\times n}$ to the following Sylvester 
matrix equation:
$$
M X_t + X_t(M - \omega I)  = E_t, \qquad 
x_t = {\rm vec}(X_t),
\quad
e_t = {\rm vec}(E_t) ,
$$
where the only nonzero element of $E_t$ is in position $(t_1,t_2)$; here the same
lexicographic order of the previous sections is used
to identify $t$ with $(t_1,t_2)$.
%To ensure existence and uniqueness of the solution, we require
%that $\omega$ be such that $\Lambda( - M) \cap \Lambda( \omega I-M) = \emptyset$.

From now on, we assume that $\Gamma = (-\infty, 0]$. 
We observe that the Sylvester equation has a unique solution, since
no eigenvalue of $M$ can be an eigenvalue of $\omega I - M$ 
for $\omega \le 0$ (recall that
$M$ is Hermitian positive definite).

%%%%%%%%%%%%%%%%%%%%%%%%%%%%%%%%%%%%%%%%%%%%%%%%%%%%%%%%%%%%%%%%%%%

%\begin{figure}
%\centering
%\includegraphics[width=3.5in,height=2.5in]{curves_Gamma.eps}
%\caption{Complex spectral plane and contours $\Gamma_1, \Gamma_2$.\label{fig:contours}}
%\end{figure}

%Let $z=z_1+\imath z_2$, with $z_1,z_2\in\RR$. 
Following Lancaster (\cite[p.556]{Lancaster1970}),
the solution matrix $X_t$ can be written as % (see \cite{Lancaster1970})
$$
X_t = 
-\int_{0}^{\infty}
%(\omega I  -( - M))^{-1} E_t (\omega I + ( z I-M))^{-1} d\omega  .
\exp(-\tau M)E_t \exp(-\tau(M - \omega I)) {\rm d}\tau .
$$
For $k=(k_1,k_2)$ and $t=(t_1,t_2)$ this gives
%$m=\lfloor (k-1)/n\rfloor+1$, $\ell=k-n \lfloor (k-1)/n\rfloor$, this gives
{%\footnotesize
\begin{eqnarray}\label{eqn:boundA}
 e_k^T(\omega I - {\cal A})^{-1}e_t %\nonumber\\
& = & e_{k_1}^T X_t e_{k_2} \nonumber\\
&=& 
-\int_{0}^{\infty}
e_{k_1}^T
\exp(-\tau M)e_{t_1} e_{t_2}^T \exp(-\tau(M - \omega I)) e_{k_2} {\rm d}\tau .
%\int_{-\infty}^{+\infty}
%%e_{k_1}^T(\omega  I  -( - M))^{-1} e_{t_1} e_{t_2}^T (\omega I + ( z I-M))^{-1}
%e_{k_1}^T
%(i\zeta I - M)^{-1} e_{t_1} e_{t_2}^T (i\zeta I + (M -\omega I))^{-1}
%e_{k_2} {\rm d}\zeta .
\end{eqnarray}
}
%where $j=\lfloor (t-1)/n\rfloor +1$, $i=t-n\lfloor (t-1)/n\rfloor $,
%$i,j\in \{1, \ldots, n\}$.
Therefore, in terms of the original matrix function component,
\begin{eqnarray*}%\label{eqn:Acomp}
 e_k^Tf({\cal A}) e_t  
=
%-\frac{1}{2\imath \pi}
%\frac{1}{4\pi^2} 
-\int_{-\infty}^0
 \int_{0}^{\infty}
e_{k_1}^T
\exp(-\tau M)e_{t_1} e_{t_2}^T \exp(-\tau(M - \omega I)) e_{k_2} {\rm d}\tau 
{\rm d}\gamma (\omega) .
%\\
%CHANGE&&=
%\frac{1}{4\pi^2} 
%\frac{1}{2\pi \imath }
%\int_{- \infty}^{+\infty}
%%e_{k_1}^T(\omega I  -(- M))^{-1} e_{t_1} e_{t_2}^T \int_\Gamma f(z) 
%%(zI + (\omega I -M))^{-1}e_{k_2} d\omega  \\
%e_{k_1}^T(i\zeta I - M)^{-1} e_{t_1} e_{t_2}^T \left ( \int_{-\infty}^0 
%(M + i\zeta I - \omega I)^{-1}e_{k_2}{\rm d}\gamma (\omega)\right ){\rm d}\zeta\\
%&&=
%%\frac{1}{4\pi^2} 
%\frac{1}{2\pi \imath }
%\int_{- \infty}^{+\infty}
%% e_{k_1}^T(\omega I  + M)^{-1} e_{t_1} e_{t_2}^T  f(M-\omega I) e_{k_2} d\omega 
%e_{k_1}^T(i\zeta I - M)^{-1} e_{t_1} e_{t_2}^T  f(M+i\zeta I) e_{k_2} {\rm d}\zeta
%  . \label{eqn:estim_Cauchy}
\end{eqnarray*}
We can thus bound each entry as
\begin{eqnarray*}
 |e_k^Tf({\cal A}) e_t|  
 \le
 \int_{0}^{\infty}
\left( |\exp(-\tau M)|_{k_1 t_1}
|\exp(-\tau M )|_{k_2 t_2}
\int_{-\infty}^0
 \exp(\tau \omega ) 
{\rm d}\gamma (\omega) \right)
{\rm d}\tau  .
\end{eqnarray*}
It is thus apparent that 
 $|e_k^Tf({\cal A}) e_t|$ can be bounded in a way analogous to the
case of Laplace--Stieltjes functions, once the term
$\int_{-\infty}^0 \exp(\tau \omega ) {\rm d}\gamma (\omega)$ is completely
determined. In particular, for $f(x) = x^{-1/2}$, 
we obtain 
\begin{eqnarray*}
\int_{-\infty}^0 \exp(\tau \omega ) {\rm d}\gamma (\omega) 
&=&
\frac{1}{\pi} \int_{-\infty}^0 \exp(\tau \omega )\frac{1}{\sqrt{-\omega}} {\rm d}\omega \\
&=&
\frac{2}{\pi} 
\int_{0}^{\infty} \exp(-\tau \eta^2 )\frac{1}{\eta} {\rm d}\eta \\
&=&
\frac{2}{\pi} \frac{\sqrt{\pi}}{2\sqrt{\tau}} =
\frac{1}{\sqrt{\pi}} f(\tau).
\end{eqnarray*}
Therefore, 
\begin{eqnarray}\label{eqn:A-1/2_CS}
\!\! |{\cal A}^{-\frac 1 2}|_{kt}
 \le
\frac{1}{\sqrt{\pi}} \!\!
\left ( \int_0^\infty \!\!\! |\exp(-\tau M)|_{k_1 t_1}^2 f(\tau) d\tau)\!\right )^{\frac 1 2}
\!\!
\left ( \int_0^\infty \!\!\! |\exp(-\tau M)|_{k_2 t_2}^2 f(\tau) d\tau)\!\right )^{\frac 1 2}.
\end{eqnarray}
Using once again the bounds in Theorem~\ref{th:boundexp} a final integral upper bound
can be obtained, in the same spirit as for Laplace--Stieltjes functions.

We explicitly mention that the solution matrix $X_t$ could be alternatively written
in terms of the resolvent $(M - \zeta i I)^{-1}$, with $\zeta\in\RR$ \cite{Lancaster1970}.
This would allow us to obtain an integral upper bound 
for $|e_k^Tf({\cal A}) e_t|$ by means of Proposition~\ref{prop:Freund} and of
(\ref{eqn:shiftedM_CS}). We omit the quite technical computations, however
the final results are qualitatively similar to those obtained above.

\begin{example}\label{ex:A-1/2_CS}
{\rm
In Figure~\ref{fig:A-1/2_CS} we report the actual decay and our estimate following
(\ref{eqn:A-1/2_CS}) for the inverse square root, again using
 the two matrices of our previous examples.
We observe that having used estimates for the exponential to handle the
Kronecker form, the approximations are slightly less sharp than previously 
seen for Cauchy--Stieltjes functions. Nonetheless, the qualitative behavior
is captured in both instances.
}
\end{example}

\begin{figure}[thb]
\centering
\includegraphics[width=2.5in,height=2.5in]{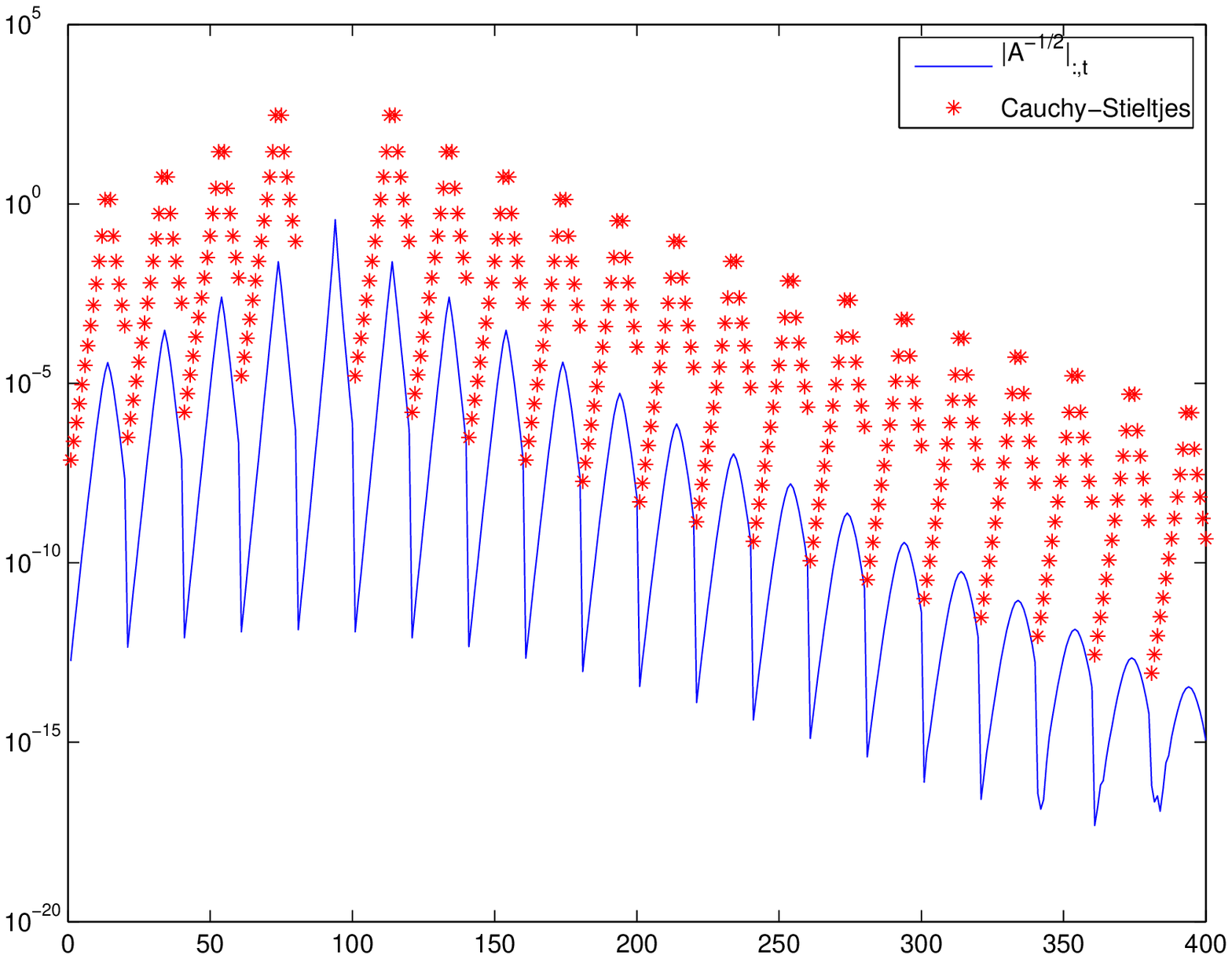}
\includegraphics[width=2.5in,height=2.5in]{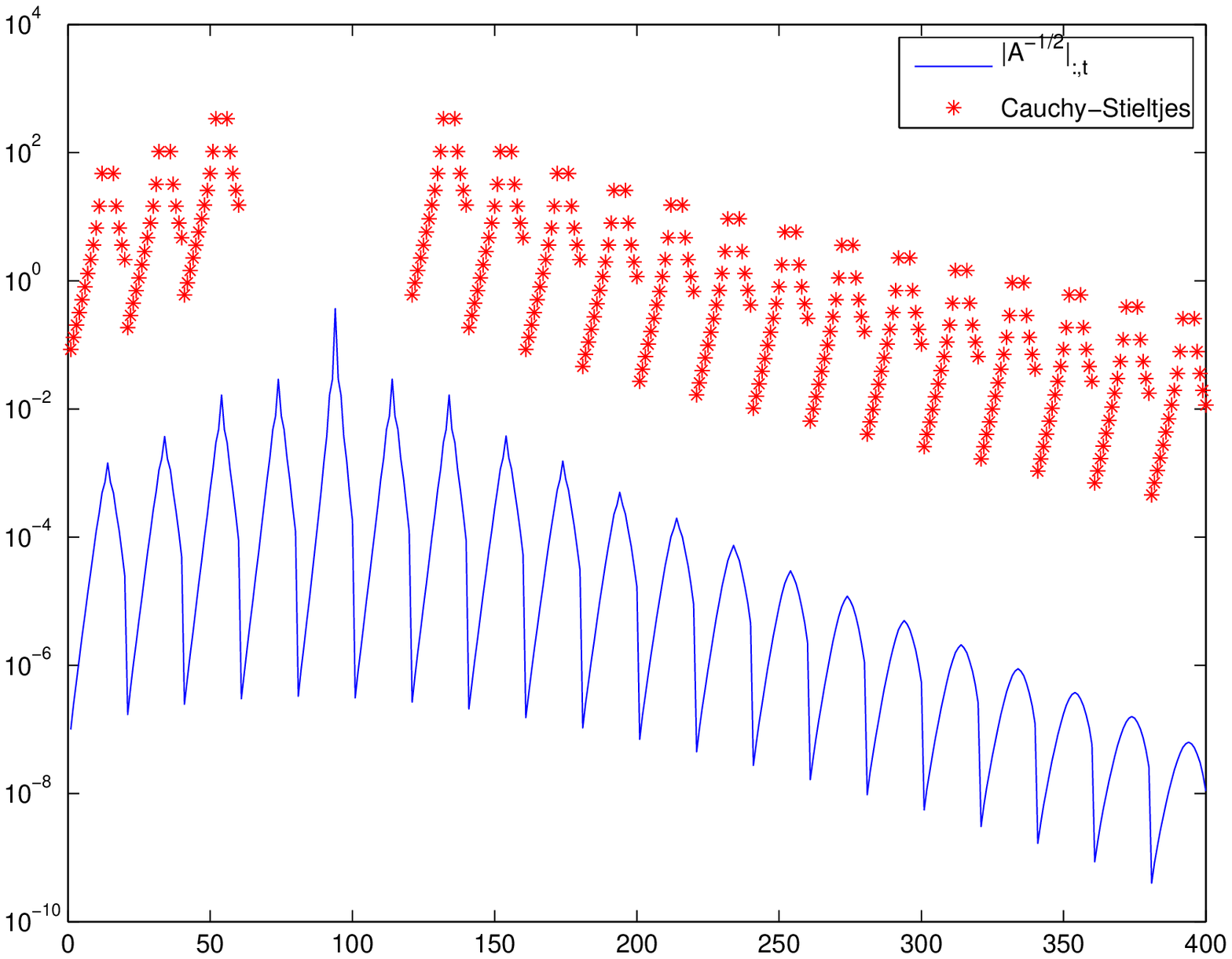}
\caption{Example \ref{ex:A-1/2_CS}.
True decay and estimates for $|A^{-\frac 1 2}|_{:,t}$, $t=94$,
$A=M\otimes I +I \otimes M$ of
size $n=400$.  Left: $M={\rm tridiag}(-1,4,-1)$.
Right: $M={\rm pentadiag}(-0.5,-1,4,-1,-0.5)$.
\label{fig:A-1/2_CS}}
\end{figure}

\begin{remark}
{\rm
As before,
the estimate for $(f({\cal A}))_{k,t}$
can be generalized to
the sum ${\cal A} = M_1 \otimes I + I \otimes M_2$, with both
$M_1, M_2$ Hermitian and positive definite matrices.
}
\end{remark}

\begin{remark}
{\rm
Using the previous remark,
the estimate for the matrix function entries
can be generalized to matrices that are
sums of several Kronecker products. For instance, if
$$
{\cal A} = 
M\otimes I\otimes I +
I\otimes M\otimes I +
I\otimes I\otimes M  ,
$$
then we can write 
$$
{\cal A} = 
M\otimes (I\otimes I) +
I\otimes (M\otimes I + I\otimes M ) =: M \otimes I + I \otimes M_2 ,
$$
so that, following the same lines as in (\ref{eqn:boundA})  we get
\begin{eqnarray*}
e_k^Tf({\cal A})e_t &=& \int_{\Gamma} e_k^T ({\cal A} - \omega I)^{-1}e_t {\rm d}\gamma (\omega) \\
& = & 
-
\int_\Gamma
\int_0^{\infty}
 e_{k_1}^T\exp(-\tau M)^{-1} e_{t_1} e_{t_2}^T \exp(-\tau(M_2-\omega I)) e_{k_2} {\rm d}\tau
{\rm d}\gamma(\omega)   .
\end{eqnarray*}
Since $M_2= M\otimes I + I\otimes M$, 
we then obtain 
$e_{t_2}^T  \exp(-\tau M_2) e_{k_2} =
e_{t_2}^T  \exp(-\tau M)\otimes \exp(-\tau M) e_{k_2}$. Splitting $t_2, k_2$ in their
one-dimensional indices, the available bounds can be employed to obtain a final 
integral estimate.
}
\end{remark}

%%%%%%%%%%%%%%%%%%%%%
\section{Conclusions} \label{sec:Conc}

In this paper we have obtained new decay bounds for the entries of certain analytic
functions of banded and sparse matrices, and used these results to obtain bounds
for functions of matrices that are Kronecker sums of banded (or sparse) matrices.
The results apply to strictly completely monotonic functions and to Markov functions,   
which include a wide variety of functions arising in mathematical
physics, numerical analysis, network science, and so forth.

The new bounds are in many cases considerably sharper than 
previously published bounds and they are able to capture the 
oscillatory, non-monotonic decay behavior observed in the
entries of $f({\cal A})$ when  $\cal A$ is a Kronecker sum.
Also, the bounds capture the superexponential decay behavior
observed in the case of entire functions.

A major difference with previous decay results is that the new bounds
are given in integral form, therefore their use requires some work 
on the part of the user. If desired, these quantities can be
further bounded for specific function choices. In practice, the
integrals can be evaluated numerically to obtain explicit bounds
on the quantities of interest.

Although in this paper we have focused mostly on the Hermitian case,
extensions to functions of more general matrices may be possible, as
long as good estimates on the entries of the matrix exponential and
resolvent are available. We leave the development of this idea for
possible future work.

\end{document}